\documentclass[11pt]{amsart}
\usepackage{amsmath,amssymb,amsfonts,url,mathptmx}
\usepackage{color}
\usepackage{appendix}
\allowdisplaybreaks[4]
\numberwithin{equation}{section}

\newtheorem{thmA}{Theorem}
\newtheorem{lemA}{Lemma}
\newtheorem{thm}{Theorem}[section]
\newtheorem{lem}{Lemma}[section]
\newtheorem{rem}{Remark}[section]
\newtheorem{prop}{Proposition}[section]
\newtheorem{cor}{Corollary}[section]

\usepackage[marginal]{footmisc}

\begin{document}

\title[Singular Mean Field Equation]{Uniqueness of bubbling solutions of mean field equations with non-quantized singularities}
\keywords{Singular mean field equations, Blow-up solutions, Singular source, Uniqueness results, asymptotic behavior}

\author{Lina Wu}\footnote{Lina Wu is partially supported by the China Scholarship Council (No.201806210165).}

\author{Lei Zhang}\footnote{Lei Zhang is partially supported by a Simons Foundation Collaboration Grant}

\address{Department of Mathematical Sciences \\
	Tsinghua University \\
	No.1 Qinghuayuan, Haidian District,
	Beijing China, 100084 }
\email{wln16@mails.tsinghua.edu.cn}

\address{Department of Mathematics\\
        University of Florida\\
        1400 Stadium Rd\\
        Gainesville FL 32611}
\email{leizhang@ufl.edu}

\date{\today}

%%%%%%%%%%%%%%%%%%%%%%%%%%%%%%%%%%%%%%%%%%%%%
\begin{abstract}
For singular mean field equations defined on a compact Riemann surface, we prove the uniqueness of bubbling solutions if some blowup points coincide with bubbling sources. If the strength of the bubbling sources at blowup points are not multiple of $4\pi$ we prove that bubbling solutions are unique under non-degeneracy assumptions. This work extends a previous work of Bartolucci, et, al \cite{bart-4}.
\end{abstract}

%%%%%%%%%%%%%%%%%%%%%%%%%%%%%%%%%%%%%%%%%%%%%
 \maketitle

\section{Introduction}

The main goal of this article is to study the uniqueness property of the following mean field equations with singularities:
\begin{equation}\label{m-equ}
\Delta_g v+\rho\bigg(\frac{he^v}{\int_M h e^v{\rm d}\mu}-\frac{1}{vol_g(M)}\bigg)=\sum_{j=1}^N 4\pi \alpha_j (\delta_{q_j}-\frac{1}{vol_g(M)}) \quad {\rm in} \ \; M,
\end{equation}
where $(M,g)$ be a Riemann surface with the metric $g$, $\Delta_g$ is the Laplace-Beltrami operator ($-\Delta_g\ge 0$), $h$ is a positive smooth function on $M$, $q_1,\cdots,q_N$ are distinct points on $M$, $\rho>0,\alpha_j>-1$ are constants, $\delta_{q_j}$ is the Dirac measure at $q_j\in M$. Equation (\ref{m-equ}) is one of the most extensively studied elliptic PDE in the past few decades, partly due to its immense and profound connections with many branches of mathematics and Physics. In conformal geometry, (\ref{m-equ}) represents a metric on M with conic singularity (see \cite{fang-lai,troy,wei-zhang-pacific}). Also it is derived from the mean
field limit of point vortices in the Euler flow \cite{caglioti-1,caglioti-2} and serves as a model equation
in the Chern-Simons-Higgs theory \cite{spruck-yang,taran-1,y-yang} and in the electroweak theory \cite{ambjorn}, etc. The literature for the study of various form of (\ref{m-equ}) is just too numerous to be listed in any reasonable way.

Recently it was found by Lin-Yan \cite{lin-yan-uniq} that the uniqueness property is particularly important for equations with concentration phenomenon. In their work \cite{lin-yan-uniq} they proved the first uniqueness property for bubbling solutions of Chern-Simon-Higgs equation and computed the exact number of solutions in certain special cases. In an important work \cite{bart-4} Bartolucci, et. al, extended Lin-Yan's result for mean field equation (\ref{m-equ}) if the blowup points are not singular sources. Our goal in this article is to further extend the uniqueness property to the case that some singular sources coincide with blowup points.

\smallskip

To write the main equation in an equivalent form, we invoke the standard Green's function$G(x,p)$:
\begin{equation}\label{gf}
\left\{\begin{array}{ll}
-\Delta_g G(x,p)=\delta_p-1\quad {\rm in}\ \; M
\\
\int_{M}G(x,p){\rm d}\mu=0,
\end{array}
\right.
\end{equation}
where the volume of $M$ is assumed to be $1$ for convenience. Then it is well known that in a neighborhood of $p$, $G(x,p)$ can be written as
$$G(x,p)=-\frac 1{2\pi }\log dist(x,p)+R(x,p)$$
where $dist(x,p)$ is the geodesic distance from $p$ to $x$ for $x$ close to $p$.

Using $G(x,p)$ we write (\ref{m-equ}) as
\begin{equation}\label{r-equ}
	\Delta_g w+\rho\bigg(\frac{He^w}{\int_M H e^w{\rm d}\mu}-1\bigg)=0 \quad {\rm in}\ \; M,
\end{equation}
where
\begin{equation}\label{r-sol}
	w(x)=v(x)+4\pi \sum_{j=1}^N \alpha_j G(x,q_j),
\end{equation}
and
\begin{equation}\label{H1}
	H(x)=h(x)\prod_{j=1}^N e^{-4\pi\alpha_j G(x,q_j)}.
\end{equation}
Note that in a local coordinate near $q_j$,
\begin{equation}\label{H2}
H(x)=h_j(x)|x-q_j|^{2\alpha_j},\quad |x-q_j|\ll 1,\quad 1\leq j\leq N,
\end{equation}
for some $h_j(x)>0$.

We say that $\{v_k\}$ is a sequence of bubbling solutions of (\ref{m-equ}) if the corresponding $w_k$ defined by (\ref{r-sol}) tends to infinity as $k$ goes to infinity. The places that $w_k$ tends to infinity are called blowup points of $v_k$ or $w_k$. In this article we use $p_1,...,p_m$ to denote blowup points. Let $q_1,...,q_N$ be the location of singular sources. If none of $p_1,...p_m$ is a singular source, Bartolucci, et. al have obtained the uniqueness of the blow up solution in \cite{bart-4}. Thus in this article we consider two cases: either all blowup points are singular sources or part of blowup points coincide with singular sources. In more precise terms let
\begin{equation}\label{pq}
\left\{\begin{array}{ll}
p_j=q_j\quad &{\rm if} \ 1\leq j \leq \tau,
\\
p_j \notin \{q_1,\cdots,q_N\}\quad &{\rm if} \ \tau+1\leq j\leq m,
\end{array}
\right.
\end{equation}
for some $1\le \tau\le m$. Thus if $\tau=m$ all blowup points are singular sources, if $1\le \tau<m$, some blowup points are singular sources and some are not. Let $4\pi\alpha_j$ be the strength of the singular source at $p_j$, so we have $\alpha_j=0$ if $j>\tau$. Since the largest $\alpha_j$ matters the most we require the first $t$ of them to have this strength:
\begin{equation}
\alpha_1=\cdots =\alpha_t>\alpha_l, \quad l\geq t+1, \quad \mbox{where } 1\le t\le \tau.
\end{equation}

It is well known that equation (\ref{r-equ}) is the Euler-Lagrange equation of the variational form:
$$I_{\rho}(w)=\frac 12 \int_M |\nabla w|^2+\rho\int_Mw-\rho \log \int_M He^w, $$
for $w\in H^1(M)$. Since adding a constant to any solution of (\ref{r-equ}) certainly gives to another solution, the space of solutions for (\ref{r-equ}) is the set of all $H^1(M)$ function with average equal to $0$. The discussion on the variational structure of (\ref{r-equ}) can be found in \cite{machiodi-1}.

\smallskip

To state the main results we use the following notations:
\begin{align}
&G_j^*(x)=8\pi (1+\alpha_j)R(x,p_j)+8\pi \sum_{l\neq j}^{1,\cdots,m}(1+\alpha_l)G(x,p_l),  \label{G_j*}
\\
&L(\mathbf{p})=\sum_{j=1}^t \big[\Delta \log h(p_j)+\rho_*-N^*-2K(p_j)\big] (h_j(p_j))^{\frac{1}{1+\alpha_1}}e^{\frac{G_j^*(p_j)}{1+\alpha_1}},   \label{L}
\\
&D(\mathbf{p})=
\begin{pmatrix}
\nabla(\log h_1+G_1^*)(p_1) \\
\cdots \\
\nabla(\log h_t+G_t^*)(p_t)
\end{pmatrix}
,  \label{D}
\end{align}
where $h_j$ is defined in (\ref{H2}), and
\begin{equation*}
  \rho_*=8\pi\sum_{j=1}^m(1+\alpha_j),\quad N^*=4\pi\sum_{j=1}^N\alpha_j
\end{equation*}

Our first result is when all blowup points are singular sources:
\begin{thm}\label{main-theorem}
	
	Let $v_k^{(1)}$ and $v_k^{(2)}$ be two sequences of bubbling solutions of (\ref{m-equ}) with $\rho_k^{(1)}=\rho_k=\rho_k^{(1)}$ and $\alpha_j\in\mathbb{R}^+\setminus\mathbb{N}\,(1\leq j\leq m)$. If $L(\mathbf{p})\neq0$ and $D(\mathbf{p})=0$, then $v_k^{(1)}=v_k^{(2)}$ for $k$ large enough.
	
\end{thm}

Note that we use $\mathbb N$ to denote the set of positive integers. The assumption that $\alpha_j\in \mathbb{R}^+\setminus\mathbb{N}$ implies that all blowup points are singular sources. It is also very essential to require $\alpha_j$ to be non-integer, since quantized singular sources ( if the strength is $4\pi N$) exhibit non-simple blowup phenomenon \cite{kuo-lin}\cite{wei-zhang-19}
that has to be studied in a separate work in the future.

The assumption of $D({\bf p})$ is also very interesting. It is well known that if $p$ is not a singular source, the vanishing rate of $D({\bf p})$ is very fast for a regular blowup point ( \cite{gluck},\cite{chen-lin-sharp}).
\medskip

Our second main result is about the uniqueness of bubbling solutions when some blowup points are non-quantized singular sources and some are regular points. So in this case we require $1\le \tau<m$
and for $(x_{\tau+1},\cdots,x_m)\in M\times\cdots \times M$, we define
\begin{equation}\label{f*}
  f^*(x_{\tau+1},\cdots,x_m)=\sum_{j=\tau+1}^m\big[\log h(x_j)+4\pi R(x_j,x_j)\big]+4\pi \sum_{l\neq j}^{\tau +1,\cdots,m}G(x_l,x_j).
\end{equation}
It is well known that $(p_{\tau+1},\cdots,p_m)$ is a critical point of $f^*$.

\begin{thm}\label{main-theorem-2}

	Let $v_k^{(1)}$ and $v_k^{(2)}$ be two sequences of bubbling solutions of (\ref{m-equ})  with $\rho_k^{(1)}=\rho_k=\rho_k^{(1)}$ and $0\leq\alpha_j<1(1\leq j\leq m)$. Suppose $1\le \tau<m$, $L(\mathbf{p})\neq0$, $D(\mathbf{p})=0$ and $\det \big(D^2f^*(p_{\tau+1},\cdots,p_m)\big)\neq 0$, then $v_k^{(1)}=v_k^{(2)}$ for $k$ large enough.
	
\end{thm}
The notation $D^2f^*$ in Theorem \ref{main-theorem-2} stands for the Hessian tensor field on $M$. Theorems \ref{main-theorem} and Theorem \ref{main-theorem-2} are clearly extensions of the main theorem in \cite{bart-4}, where the uniqueness of bubbling solutions around regular blowup points is established. Here in our work, the assumptions of $L(\mathbf{p})$ and $D(\mathbf{p})$
are only placed on singular sources with the strongest strength.

In addition to the importance of application, the proof of the main theorems requires extremely delicate local analysis, just like the argument in \cite{bart-4}. Our argument relies heavily on the result of the second author in \cite{zhang2}, Chen-Lin's refined estimates in \cite{chen-lin,chen-lin-deg-2} and the argument used by Lin-Yan \cite{lin-yan-uniq} and Bartolucci-et-al \cite{bart-4}. Even though the outline of our paper is similar to those used in \cite{lin-yan-uniq,bart-4} we have to establish accurate estimates for certain terms in an iterative manner.

The proof of Theorem \ref{main-theorem} and Theorem \ref{main-theorem-2} can also be applied to solve the following locally defined Dirichlet boundary problem:
Let $\Omega$ be an open and bounded domain in $\mathbb{R}^2$ with regular boundary $\partial\Omega\in C^2$, $v$ be a solution of
\begin{equation}\label{equ-flat}
\left\{\begin{array}{lll}
\Delta v+\rho \frac{he^v}{\int_{\Omega} h e^v{\rm d}x}=\sum_{j=1}^N 4\pi \alpha_j \delta_{q_j}  &{\rm in} \;\ \Omega,
\\
v=0  &{\rm on} \;\ \partial\Omega,
\end{array}
\right.
\end{equation}
where $h>0$ is a $C^1$ funation in $\Omega$, $q_1,\cdots,q_N$ are distinct points in $\Omega$, $\rho>0$, $\alpha_j>0$ are constants.

Let $\{v_k\}$ be a sequence of solutions to (\ref{equ-flat}) with $\rho=\rho_k$. We say
\begin{equation}\label{blowup-flat}
v_k \  {\rm blows} \  {\rm up} \  {\rm at} \  p_j\in\Omega,\quad 1\leq j\leq m,
\end{equation}
if $\rho \frac{he^v}{\int_{\Omega} h e^v{\rm d}x}\rightharpoonup8\pi\sum_{j=1}^N(1+\alpha_j)\delta_{p_j}$ in $\Omega$ in the sense of measure, where $\alpha_j=0$ if $p_j\notin\{q_1\cdots,q_N\}$. Similar to notations for the first part, we assume there exist $1\leq t\leq\tau\leq m$ such that $\alpha_1=\cdots\alpha_t>\alpha_i$, $i\ge t+1$ and $\alpha_{\tau+1}=\cdots\alpha_m$.

Let $G_{\Omega}$ be the Green's function defined by
\begin{equation*}
\left\{\begin{array}{lll}
-\Delta G_{\Omega}(x,p)=\delta_{p}  &{\rm in} \;\ \Omega,
\\
G_{\Omega}(x,p)=0  &{\rm on} \;\ \partial\Omega,
\end{array}
\right.
\end{equation*}
and $R_{\Omega}(x,p)=G_{\Omega}(x,p)+\frac{1}{2\pi}\log |x-p|$ be the regular part of $G_{\Omega}(x,p)$. In order to state the uniqueness results of (\ref{equ-flat}) we denote $N^*=4\pi\sum_{j=1}^m\alpha_j$ and
\begin{align*}
&G_{j,\Omega}^*(x)=8\pi (1+\alpha_j)R_{\Omega}(x,p_j)+8\pi \sum_{l\neq j}^{1,\cdots,m}(1+\alpha_l)G_{\Omega}(x,p_l),
\\
&L_{\Omega}(\mathbf{p})=\sum_{j=1}^t \big[\Delta \log h(p_j)-N^*\big] (h_j(p_j))^{\frac{1}{1+\alpha_1}}e^{\frac{G_j^*(p_j)}{1+\alpha_1}},
\\
&D_{\Omega}(\mathbf{p})=
\begin{pmatrix}
\nabla(\log h_1+G_1^*)(p_1) \\
\cdots \\
\nabla(\log h_t+G_t^*)(p_t)
\end{pmatrix}
.
\end{align*}

Then we have the following result similar to Theorem \ref{main-theorem}.

\begin{thm}\label{main-theorem-3}
	
	Let $v_k^{(1)}$ and $v_k^{(2)}$ be two sequences of solutions of (\ref{equ-flat}) (\ref{blowup-flat}) with $\rho_k^{(1)}=\rho_k=\rho_k^{(1)}$ and $\alpha_j\in\mathbb{R}^+\setminus\mathbb{N}\,(1\leq j\leq m)$. If $L_{\Omega}(\mathbf{p})\neq0$ and $D_{\Omega}(\mathbf{p})=0$, then $v_k^{(1)}=v_k^{(2)}$ for $k$ large enough.
	
\end{thm}

If the set of blowup points is a mixture of non-quantized singular sources and regular points, we also have a uniqueness result. Let
\begin{equation*}
f_{\Omega}^*(x_{\tau+1},\cdots,x_m)=\sum_{j=\tau+1}^m\big[\log h(x_j)+4\pi R(x_j,x_j)\big]+4\pi \sum_{l\neq j}^{\tau +1,\cdots,m}G(x_l,x_j),
\end{equation*}
and $D^2f_{\Omega}^*$ be the Hessian tensor field on $M$. In this case, $(p_{\tau+1},\cdots,p_m)$ is a critical point of $f_{\Omega}^*$. Then, we obtain the following result.

\begin{thm}\label{main-theorem-4}

	Let $v_k^{(1)}$ and $v_k^{(2)}$ be two sequences of solutions of (\ref{equ-flat}) (\ref{blowup-flat}) with $\rho_k^{(1)}=\rho_k=\rho_k^{(1)}$ and $0\leq\alpha_j<1(1\leq j\leq m )$. If $L_{\Omega}(\mathbf{p})\neq0$, $D_{\Omega}(\mathbf{p})=0$ and $\det \big(D^2f_{\Omega}^*(p_{\tau+1},\cdots,p_m)\big)\neq 0$, then $v_k^{(1)}=v_k^{(2)}$ for $k$ large enough.
	
\end{thm}

When we were in the final stage of writing this article, we found that Bartolucci, et, al \cite{bart-4-2} posted an article on arxiv.org about the same topic. Their theorem is a special case of our results and both works were carried out independently.

\smallskip

The organization of this paper is as follows. Section \ref{preliminary} is dedicated to notations and preliminary sharp estimates for bubbling solutions of equation (\ref{m-equ}).  In section \ref{difference} we consider the differences between two bubbling sequences and establish many estimates near each blowup point and away from all blowup points. In section \ref{anal-pohozaev} we derive some Pohozaev-type identities and evaluate each term carefully. These Pohozaev identities play a key role in the proof of the main theorems. Finally the proof of Theorem \ref{main-theorem} is placed in section \ref{pf-uni-1} and that of Theorem \ref{main-theorem-2} can be found in section \ref{pf-uni-2}. At the end of section \ref{pf-uni-2}, we list the brief sketch of the proof of Theorems \ref{main-theorem-3} and \ref{main-theorem-4} based on well known facts \cite{ma-wei}.

\section{Preliminary Estimates}\label{preliminary}

Since the proof of the main theorems requires very delicate analysis, in this section we list some established estimates in \cite{chen-lin-sharp,chen-lin,zhang1,zhang2}.

Let $w_k$ be a sequence of solutions of (\ref{r-equ}) with $\rho =\rho_k$. Suppose that $w_k$ blows up at $m$ points $\{p_1 \cdots,p_m\}$ as we have stated in section one. To describe the bubbling profile of $w_k$ near $p_j$, we set
\begin{equation}\label{n-sol}
	u_k=w_k-\log\bigg(\int_M He^{w_k}{\rm d}\mu \bigg)
\end{equation}
 and write the equation for $u_k$ as
\begin{equation}\label{n-equ}
\Delta_g u_k+\rho_k(He^{u_k}-1)=0\quad {\rm in} \ \; M.
\end{equation}
It is easy to observe from the definition of $u_k$ that $$\int_{M}He^{u_k}{\rm d}\mu=1.$$

From previous works of Liouville equations ( for example \cite{chen-lin-sharp} ),
\begin{equation}\label{local-cov}
 u_k-\bar{u}_k \ \to \sum_{j=1}^m 8\pi(1+\alpha_j)G(x,p_j) \quad {\rm in} \ \; {\rm C}_{loc}^2(M\backslash \{p_1,\cdots,p_m\})
\end{equation}
where $\bar{u}_k$ is the average of $u_k$ on $M$:
$$\bar{u}_k=\int_{M}u_k{\rm d}\mu.$$

For the convenience later we fix $r_0>0$ small  and $M_j\subset M, 1\leq j\leq m$ such that
\begin{equation}\label{Mj}
M=\bigcup_{j=1}^m \overline{M}_j;\quad M_j\cap M_l=\varnothing,\  {\rm if}\ j\neq l;\quad B(p_j,3r_0)\subset M_j, \ j=1,\cdots,m.
\end{equation}
According to this definition $M_1=M$, if $m=1$.

Then we use  $\lambda_{k,j}$ to denote
\begin{equation}\label{lambda_kj}
\lambda_{k,j}=\left\{
\begin{array}{lcl}
u_k(p_j) && {\rm if}\ \,\alpha_j\neq 0, \\
u_k(p_{k,j}) \mathrel{\mathop:}=\max_{B(p_j,r_0)}u_k  && {\rm if}\ \,\alpha_j= 0.
\end{array}
\right.
\end{equation}
and let $U_{k,j}$ be a global solution of
\begin{equation}\label{U_kj-equ}
\Delta U_{k,j}+\rho_kh_j(p_{k,j})|x-p_{k,j}|^{2\alpha_j}e^{U_{k,j}}=0 \quad  {\rm in} \ \; \mathbb{R}^2
\end{equation}
with the expression ($U_{k,j}$ is called a standard bubble):
\begin{equation}\label{U_kj}
U_{k,j}(x)=\lambda_{k,j}-2\log\Big(1+\frac{\rho_k h_j(p_{k,j})}{8(1+\alpha_j)^2}e^{\lambda_{k,j}}|x-p_{k,j}|^{2(1+\alpha_j)}\Big).
\end{equation}

It is well-known \cite{li-cmp,Bartolucci-Chen-Lin-T} that $u_k$ can be approximated by the standard bubbles $U_{k,j}$ near $p_j$ with $O(1)$ error:
\begin{equation}\label{standard-bubble}
  \big|u_k(x)-U_{k,j}(x)\big| \leq C, \quad x\in B(p_j,r_0).
\end{equation}

 As a consequence,
\begin{equation}\label{lambda-ij}
|\lambda_{k,i}-\lambda_{k,j}|\leq C, \quad 1\leq i,j \leq m.
\end{equation}
for some $C$ independent of $k$.
Furthermore, it is established in \cite{bart-taran-mass} that $\rho_*=\lim_{k\to +\infty}\rho_k$.

\medskip

Later, sharper estimates were obtained in \cite{zhang2,chen-lin} for $1\leq j \leq \tau$ and in \cite{chen-lin-sharp,zhang1,gluck} for $\tau+1\leq j\leq m$. In order to apply those estimates, we might consider the equation in terms of the flat metric and introduce the following notations.

\medskip

In $B(p_j,r_0)$, the flat metric is ${\rm d}s^2=e^{\phi_j}\big(({\rm d}x_1)^2+({\rm d}x_2)^2\big)$ with $\phi_j$ satisfying
\begin{equation}\label{phi-equ}
\left\{\begin{array}{lcl}
\Delta \phi_j+2Ke^{\phi_j}=0, &&  {\rm in} \ \; B(p_j,r_0),
\\
\phi_j(0)=|\nabla \phi_j(0)|=0, && \quad
\end{array}
\right.
\end{equation}
where $0$ is the coordinate of $p_j$, $\Delta =\sum _{i=1}^{2}\frac{\partial^2}{\partial x_i^2}$. In this local coordinate, equation (\ref{n-equ}) is equivalent to
\begin{equation}\label{flat-equ}
\Delta u_k+\rho_ke^{\phi_j}(He^{u_k}-1)=0\quad {\rm in} \ \; B(p_j,r_0).
\end{equation}
If we denote $\tilde{h}_j=h_je^{\phi_j}$, then (\ref{flat-equ}) can be written as follows:
\begin{equation}\label{flat-equ-2}
\Delta u_k+\rho_k\tilde{h}_j|x-p_j|^{2\alpha_j}e^{u_k}-\rho_ke^{\phi_j}=0\quad {\rm in} \ \; B(p_j,r_0).
\end{equation}

To state the more refined asymptotic analysis we introduce the following notations:
\begin{equation}\label{loc-mass}
\rho_{k,j}=\int_{B(p_{k,j},r_0)}\rho_kHe^{u_k}{\rm d}\mu,\quad 1\leq j\leq m,
\end{equation}
\begin{equation}\label{sigma1}
\sigma_k(x)=u_k(x)-\bar{u}_k-\sum_{j=1}^m\rho_{k,j} G(x,p_{k,j}), \quad x\in M\backslash \bigcup_{j=1}^mB(p_{k,j},\frac{r_0}{2}),
\end{equation}
\begin{equation}\label{G_kj}
G_{k,j}(x)=\rho_{k,j}R(x,p_{k,j})+\sum_{l\neq j}^{1,\cdots,m}\rho_{k,l}G(x,p_{k,l}),\quad x\in B(p_{k,j},r_0),
\end{equation}
where $R(x,p_{k,j})$ is the regular part of $G(x,p_{k,j})$. Finally
for $x\in B(p_{k,j},r_0)$, set
\begin{equation*}
\tilde{u}_{k,j}(x)=u_k(x)-\big(G_{k,j}(x)-G_{k,j}(p_{k,j})\big),
\end{equation*}
\begin{equation}\label{eta1}
\eta_{k,j}(x)=\tilde{u}_{k,j}(x)-U_{k,j}(x).
\end{equation}

\subsection{Sharper estimates}
\quad
\medskip

If $\alpha_j\in \mathbb{R}^+\setminus\mathbb{N}$, in order to obtain the refined estimates of the bubbling solutions, the second author considered the harmonic function $\psi_{k,j}$ in \cite{zhang2}, which satisfies
\begin{equation}\label{psi-equ}
\left\{\begin{array}{lcl}
\Delta \psi_{k,j}=0 && {\rm in} \ \; B(p_{k,j},r_0),
\\
\psi_{k,j} =\tilde{u}_{k,j}-\frac{1}{2\pi r_0}\int_{\partial B(p_{k,j},r_0)}\tilde{u}_{k,j} {\rm d}s && {\rm on} \ \; \partial B(p_{k,j},r_0).
\end{array}
\right.
\end{equation}

With the help of $\psi_{k,j}$, Zhang and Chen-Lin proved the following sharp estimate in \cite{zhang2}.
\begin{thmA}\label{Theorem zhang}\cite{zhang2,chen-lin}
	
	For $x\in B(p_{k,j},r_0)$, it holds that
	\begin{equation}\label{zhang}
	\begin{split}
	\eta_{k,j}(x)=&\psi_{k,j}(x)-\frac{2(1+\alpha_j)}{\alpha_j}\frac{\langle a,x-p_{k,j}\rangle}{1+\frac{\rho_k h_j(p_{(k,j)})}{8{(1+\alpha_j)^2}}e^{\lambda_{k,j}}|x-p_{k,j}|^{2(1+\alpha_j)}}\\
	& +d_j\log \big(2+e^{\frac{\lambda_{k,j}}{2(1+\alpha_j)}}|x-p_{k,j}|\big)e^{-\frac{\lambda_{k,j}}{1+\alpha_j}}+O(e^{-\frac{\lambda_{k,j}}{1+\alpha_j}}),
	\end{split}
	\end{equation}
	where $a=\nabla(\log h_j+G_{k,j})(p_{k,j}) \in \mathbb{R}^2$ and
	\begin{equation*}
		d_j=\frac{\pi}{(1+\alpha_j)\sin \frac{\pi}{1+\alpha_j}}\Big(\frac{8(1+\alpha_j)^2}{\rho_kh_j(p_{k,j})}\Big)^{\frac{1}{1+\alpha_j}}\big[\Delta\log h(p_j)+\rho_*-N^*-2K(p_j)\big].
	\end{equation*}
\end{thmA}

In \cite{chen-lin},the following estimates for $\psi_{k,j}$ and $\sigma_k$ are established:
\begin{thmA}\label{Theorem chen-lin1}\cite{chen-lin}
	\begin{equation}\label{psi-est}
	|\psi_{k,j}(x)|=O(e^{-\frac{\lambda_{k,j}}{1+\alpha_1}}),\quad x\in B(p_{k,j},r_0).
	\end{equation}
	\begin{equation}\label{sigma2}
	|\sigma_k(x)|+|\nabla\sigma_k(x)|=O(e^{-\frac{\lambda_{k,1}}{1+\alpha_1}}),\quad x\in M\backslash \big(\bigcup_{j=1}^mB(p_{k,j},\frac{r_0}{2})\big).
	\end{equation}
\end{thmA}

Then, by Theorem \ref{Theorem zhang} and Theorem \ref{Theorem chen-lin1}, we have
\begin{equation}\label{eta2}
|\eta_{k,j}(x)|=O(e^{-\frac{\lambda_{k,j}}{2(1+\alpha_j)}}+e^{-\frac{\lambda_{k,j}}{1+\alpha_1}}),\quad x\in B(p_{k,j},r_0),\quad 1\leq j\leq \tau.
\end{equation}
For the case $\tau+1\leq j \leq m $, the estimate for $\eta_{k,j}$, established in \cite{chen-lin-sharp}\cite{zhang1}\cite{gluck}, is
\begin{equation}\label{eta3}
|\eta_{k,j}(x)|=O(\lambda_{k,j}e^{-\lambda_{k,j}}),\quad x\in B(p_{k,j},r_0),\quad \tau+1\leq j\leq m.
\end{equation}

Moreover, according to the proof of Theorem 3.5 in \cite{chen-lin}, the following estimate holds:
\begin{equation}\label{uk-ave-1}
\bar{u}_k+\lambda_{k,j}+2\log\dfrac{\rho_kh_j(p_{k,j})}{8(1+\alpha_j)^2}+G_{k,j}(p_{k,j})+\frac{d_j}{2(1+\alpha_j)}\lambda_{k,j} e^{-\frac{\lambda_{k,j}}{1+\alpha_1}}=O(e^{-\frac{\lambda_{k,j}}{1+\alpha_1}}).
\end{equation}
As a consequence, we have
\begin{equation}\label{uk-ave-2}
\lambda_{k,j}-\lambda_{k,1}=2\log\dfrac{(1+\alpha_j)^2h_1(p_{k,1})}{(1+\alpha_1)^2 h_j(p_{k,j})}+G_{k,1}(p_{k,1})-G_{k,j}(p_{k,j})+O(e^{-\frac{\lambda_{k,1}}{2(1+\alpha_1)}}).
\end{equation}

For the difference between $\rho_k$ and $\rho_*$, $\rho_k$ and $8\pi(1+\alpha_j)$, the following estimates also have been proved in \cite{chen-lin,chen-lin-sharp}.
\begin{thmA}\label{Theorem chen-lin2}\cite{chen-lin,chen-lin-sharp}	
	\begin{align}
	&\rho_{k,j}-8\pi(1+\alpha_j)=2\pi d_je^{-\frac{\lambda_{k,j}}{1+\alpha_j}}+O\big(e^{-\frac{1+\gamma}{1+\alpha_1}\lambda_{k,1}}\big), && 1\leq j\leq \tau, \label{rho-kj-1}
	\\
	&\rho_{k,j}-8\pi=O\big(\lambda_{k,j}e^{-\lambda_{k,j}}\big), &&\tau+1\leq j\leq m,  \label{rho-kj-2}
	\\
	&\rho_k-\rho_*=L^*e^{-\frac{\lambda_{k,1}}{1+\alpha_1}}+O\big(e^{-\frac{1+\gamma}{1+\alpha_1}\lambda_{k,1}}\big),&&  \label{rho-k}
	\end{align}
	with fixed $\gamma\in (0,\min({\alpha_1,\frac{1}{2}}))$ small and
	$$L^*=\dfrac{2\pi^2}{(1+\alpha_1)\sin\frac{\pi}{1+\alpha_1}}e^{-\frac{G_1^*(p_{1})}{1+\alpha_1}}\Big(\frac{8(1+\alpha_1)^2}{\rho_* h_1(p_1)^2}\Big)^{\frac{1}{1+\alpha_1}}L(\mathbf{p}).$$
\end{thmA}

\smallskip

If $\tau<m$, as in \cite{chen-lin-sharp} and \cite{bart-4}, non-degeneracy condition $\det\big(D^2f^*(p_{\tau+1},\cdot,p_m)\big)\neq0$ leads to
\begin{equation}\label{p_kj-location}
|p_{k,j}-p_j|=O(\lambda_{k,j}e^{-\lambda_{k,j}}),\quad \tau+1\leq j\leq m.
\end{equation}
Futhermore, in \cite{chen-lin}, the authors showed that
\begin{equation}\label{first-deriv-est}
\nabla(\log h+G_j^*)(p_{k,j})=O(e^{-\frac{\lambda_{k,j}}{1+\alpha_1}}),\quad \tau+1\leq j\leq m.
\end{equation}

\subsection{The kernel of the linearized equations}
\quad
\medskip

In the proof of the uniqueness, we need some facts about the linearized equation after the appropriate rescale.

\medskip

For $\alpha\in \mathbb{R}^+\setminus\mathbb{N}$, Chen-Lin proved the following lemma in \cite{chen-lin}.
\begin{lemA}\label{linear-lem-1}
	Suppose $ \alpha>0 $ is not an integer, $ \varphi $ is a $C^2$-function that satisfies
	\begin{equation*}
		\left\{\begin{array}{ll}
		\Delta \varphi+|x|^{2\alpha}e^{U_\alpha}\varphi=0\quad & {\rm in} \ \; \mathbb{R}^2,
		\\
		|\varphi| \leq (1+|x|)^{\kappa} \quad & {\rm in} \ \;\mathbb{R}^2,
		\end{array}
		\right.
	\end{equation*}
	where $ U_{\alpha}(x)=\log\frac{8(1+\alpha)^2}{(1+|x|^{2(1+\alpha)})^2} $ and $\kappa\in(0,1)$. Then there exists some constant $b_0$ such that
	\begin{equation*}
	\varphi(x)= b_0\frac{1-|x|^{2(1+\alpha)}}{1+|x|^{2(1+\alpha)}}.
	\end{equation*}
\end{lemA}

For $\alpha=0$, Chen-Lin proved the following lemma in \cite{chen-lin-sharp}.
\begin{lemA}\label{linear-lem-2}
	Let  $ \varphi $ be a $ C^2 $-function of
	\begin{equation*}
	\left\{\begin{array}{ll}
	\Delta \varphi+e^U\varphi=0\quad & {\rm in} \ \;\mathbb{R}^2,
	\\
	|\varphi| \leq c\big(1+|x|\big)^{\kappa} \quad & {\rm in} \ \;\mathbb{R}^2,
	\end{array}
	\right.
	\end{equation*}
	where $ U(x)=\log\frac{8}{(1+|x|^2)^2} $ and $ \kappa \in[0,1) $. Then there exist constants $b_0$, $b_1$, $b_2$ such that
	\begin{equation*}
		 \varphi= b_0\varphi_0+b_1\varphi_1+b_2\varphi_2,
	\end{equation*}
	where
	\begin{equation*}
	\varphi_0(x)= \frac{1-|x|^2}{1+|x|^2},\quad \varphi_1(x)= \frac{x_1}{1+|x|^2},\quad \varphi_2(x)= \frac{x_2}{1+|x|^2}.
	\end{equation*}
	
\end{lemA}

\section{The difference between $ u_k^{(1)} $ and $ u_k^{(2)} $}\label{difference}

The way we prove the main theorems is by contradiction. So we assume that $ u_k^{(1)} $ and $ u_k^{(2)} $
are two different sequences of solutions to (\ref{r-equ}) with $\rho_k^{(1)}=\rho_k=\rho_k^{(2)}$, and common blowup points located at $p_1,\cdots,p_m$. For $i=1,2$, we use
the following notations
$$\lambda_{k,j}^{(i)}, u_{k,j}^{(i)}, v_{k,j}^{(i)}, \rho_{k,j}^{(i)}, \bar{u}_k^{(i)}, U_{k,j}^{(i)}, G_{k,j}^{{(i)}}, \psi_{k,j}^{(i)}, \eta_{k,j}^{(i)}, \epsilon_{k,j}^{(i)},  \sigma_k^{(i)}, p_{j}^{(i)}, $$
with obvious interpretations in the context.
Finally the following three functions are defined by the difference of $u_1^k$ and $u_2^k$:
\begin{align}
&	\varsigma_k(x)=\dfrac{u_k^{(1)}(x)-u_k^{(2)}(x)}{\parallel u_k^{(1)}-u_k^{(2)}\parallel_{L^{\infty}(M)}}, \label{varsigma}
\\
&f_k(x)=\rho_k H(x)\frac{e^{u_k^{(1)}(x)}-e^{u_k^{(2)}(x)}}{\parallel u_k^{(1)}-u_k^{(2)}\parallel_{L^{\infty}(M)}}, \label{f}
\\
&c_k(x)=\dfrac{e^{u_k^{(1)}(x)}-e^{u_k^{(2)}(x)}}{u_k^{(1)}(x)-u_k^{(2)}(x)}.  \label{c}
\end{align}
Clearly $\varsigma_k$ satisfies
\begin{equation}\label{sigma-equ}
\Delta_g\varsigma_k(x)+f_k(x)=\Delta_g\varsigma_k(x)+\rho_kH(x)c_k(x)\varsigma_k(x)=0,\quad x\in M.
\end{equation}

As the first step of our proof, we give an initial estimate of $ \parallel u_k^{(1)}-u_k^{(2)}\parallel_{L^{\infty}(M)}$ using $L({\bf p})\neq 0$:
\begin{lem}\label{est1}
Under the assumption of $L(\mathbf{p})\neq 0$, we have
\begin{equation}\label{u-est1}
\parallel u_k^{(1)}-u_k^{(2)}\parallel_{L^{\infty}(M)}=O(e^{-\frac{\gamma}{1+\alpha_1}\lambda_{k,1}^{(1)}}).
\end{equation}

\end{lem}

\begin{proof}[\textbf{Proof}]
		
\textbf{Step 1.}
For $x\in B(p_{k,j},r_0), 1\leq j\leq m$, by (\ref{eta1}) (\ref{eta2}) (\ref{loc-mass}) (\ref{G_kj}) and Theorem \ref{Theorem chen-lin2}, we have
\begin{align*}
&u_k^{(1)}(x)-u_k^{(2)}(x)\\
=&\,U_{k,j}^{(1)}(x)-U_{k,j}^{(2)}(x)+\eta_{k,j}^{(1)}(x)-\eta_{k,j}^{(2)}(x)+G_{k,j}^{(1)}(x)-G_{k,j}^{(2)}(x)\\
&\,+G_{k,j}^{(1)}(p_{k,j}^{(1)})-G_{k,j}^{(2)}(p_{k,j}^{(2)})\\
=&\,\lambda_{k,j}^{(1)}-\lambda_{k,j}^{(2)}-2\log\Big(1+\frac{\rho_k h_j(p_{k,j}^{(1)})}{8(1+\alpha_j)^2}e^{\lambda_{k,j}^{(1)}}\big|x-p_{k,j}^{(1)}\big|^{2(1+\alpha_j)}\Big)\\
&\,+2\log\Big(1+\frac{\rho_k h_j(p_{k,j}^{(2)})}{8(1+\alpha_j)^2}e^{\lambda_{k,j}^{(2)}}\big|x-p_{k,j}^{(2)}\big|^{2(1+\alpha_j)}\Big)+O\Big(\sum_{i=1}^{2}e^{-\frac{\lambda_{k,1}^{(i)}}{2(1+\alpha_1)}}\Big).
\end{align*}
Theorem \ref{Theorem chen-lin2} and $L(\mathbf{p})\neq 0$ give rise to
\begin{equation*}
e^{-\frac{1}{1+\alpha_1}(\lambda_{k,1}^{(1)}-\lambda_{k,1}^{(2)})}=1+O\Big(\sum_{i=1}^{2}e^{-\frac{\gamma}{1+\alpha_1}\lambda_{k,1}^{(i)}}\Big),
\end{equation*}
which immediately implies
\begin{equation}\label{lam-est-1}
\lambda_{k,1}^{(1)}-\lambda_{k,1}^{(2)}=O\Big(\sum_{i=1}^{2}e^{-\frac{\gamma}{1+\alpha_1}\lambda_{k,1}^{(i)}}\Big).
\end{equation}
Then by (\ref{lam-est-1}) and (\ref{uk-ave-2}), what holds for one point is also true at other blowup points:
\begin{equation}\label{lam-est-2}
\lambda_{k,j}^{(1)}-\lambda_{k,j}^{(2)}=O\Big(\sum_{i=1}^{2}e^{-\frac{\gamma}{1+\alpha_1}\lambda_{k,1}^{(i)}}\Big),\quad 1\leq j\leq m.
\end{equation}

On the other hand, using (\ref{p_kj-location}) in direct computation, we have,
\begin{align*}
&\log\Big(1+\frac{\rho_k h_j(p_{k,j}^{(1)})}{8(1+\alpha_j)^2}e^{\lambda_{k,j}^{(1)}}\big|x-p_{k,j}^{(1)}\big|^{2(1+\alpha_j)}\Big)-\log\Big(1+\frac{\rho_k h_j(p_{k,j}^{(2)})}{8(1+\alpha_j)^2}e^{\lambda_{k,j}^{(2)}}\big|x-p_{k,j}^{(2)}\big|^{2(1+\alpha_j)}\Big)\\
&=O(\lambda_{k,j}^{(1)}-\lambda_{k,j}^{(2)})
\end{align*}
Thus $u_k^{(1)}$ and $u_k^{(2)}$ are close in the interior of the ball $B(p_{k,j}^{(1)},r_0)$:
\begin{equation}\label{est-1}
\parallel u_k^{(1)}-u_k^{(2)}\parallel_{L^{\infty}(B(p_{k,j}^{(1)},r_0))}=O\Big(\sum_{i=1}^{2}e^{-\frac{\gamma}{1+\alpha_1}\lambda_{k,1}^{(i)}}\Big)=O(e^{-\frac{\gamma}{1+\alpha_1}\lambda_{k,1}^{(1)}}).
\end{equation}

\textbf{Step 2.}
For $x\in M\backslash \bigcup_{j=1}^mB(p_{k,j}^{(1)},r_0)$, we first use the Green's representation formula to write $u_k^{(1)}-u_k^{(2)}-\big(\bar{u}_k^{(1)}-\bar{u}_k^{(2)}\big)$ in three parts:

\begin{equation}
\begin{split}
 & u_k^{(1)}-u_k^{(2)}-\big(\bar{u}_k^{(1)}-\bar{u}_k^{(2)}\big) \notag\\
=& \int_{M} G(y,x)\rho_kH(y)(e^{u_k^{(1)}(y)}-e^{u_k^{(2)}(y)}){\rm d}\mu(y) \\
=& \sum_{j=1}^{m}\int_{B(p_{k,j}^{(1)},\frac{r_0}{2})} \big(G(y,x)-G(p_{k,j}^{(1)},x)\big)\rho_kH(y)(e^{u_k^{(1)}(y)}-e^{u_k^{(2)}(y)}){\rm d}\mu(y) \\
&+ \sum_{j=1}^{m}G(p_{k,j}^{(1)},x)\int_{B(p_{k,j}^{(1)},\frac{r_0}{2})}\rho_kH(y)(e^{u_k^{(1)}(y)}-e^{u_k^{(2)}(y)}){\rm d}\mu(y)  \\
&+ \int_{M\backslash \bigcup_{j=1}^mB(p_{k,j}^{(1)},\frac{r_0}{2})} G(y,x)\rho_kH(y)(e^{u_k^{(1)}(y)}-e^{u_k^{(2)}(y)}){\rm d}\mu(y) \\
=&\mathrel{\mathop:}I_1+I_2+I_3.
\end{split}
\end{equation}
Before we evaluate each one of them we recall a few facts: First
\begin{equation*}
	p_{k,j}^{(1)}-p_{k,j}^{(2)}=\left\{\begin{array}{ll}0, \quad \mbox{for}\  1\le j\le \tau, \\
	O(\sum_{i=1}^2\lambda_{k,j}^{(i)}e^{-\lambda_{k,j}^{(i)}}) \quad \mbox{if}\  j>\tau \ \mbox{ (see (\ref{p_kj-location}))}.
	\end{array}
	\right.
\end{equation*}
Next for $x\in M\backslash \bigcup_{j=1}^mB(p_{k,j}^{(1)},r_0)$, $ y\in B(p_{k,j}^{(1)},\frac{r_0}{2}) $,
\begin{equation*}
	G(y,x)-G(p_{k,j}^{(1)},x)=\langle\partial_yG(y,x)\big|_{y-p_{k,j}^{(1)}},y-p_{k,j}^{(1)}\rangle+O(|y-p_{k,j}^{(1)}|^2)
\end{equation*}
Then using symmetry, scaling, and the closeness between $u_k^{(i)}$ with standard bubbles, we have

\begin{align*}
&I_1
=\sum_{j=1}^{m}\sum_{i=1}^2\int_{B(p_{k,j}^{(i)},\frac{r_0}{2})}\frac{\langle\partial_yG(y,x)\big|_{y=p_{k,j}^{(i)}},y-p_{k,j}^{(i)}\rangle\rho_k\tilde{h}_j(y)|y-p_{k,j}^{(i)}|^{2\alpha_j}}{\big(1+\frac{\rho_kh_j(p_{k,j}^{(i)})}{8(1+\alpha_j)^2}e^{\lambda_{k,j}^{(i)}}|y-p_{k,j}^{(1)}|^{2(1+\alpha_j)}\big)^2}\\
&\times \Big(1+O(|y-p_{k,j}^{(i)}|)+O(e^{-\frac{\lambda_{k,j}^{(i)}}{2(1+\alpha_j)}})+O(e^{-\frac{\lambda_{k,j}^{(i)}}{1+\alpha_1}})\Big){\rm d}y+
 O\Big(\sum_{i=1}^{2}e^{-\frac{\lambda_{k,1}^{(i)}}{1+\alpha_1}}\Big),\\
 &=O\Big(\sum_{i=1}^{2}e^{-\frac{\lambda_{k,1}^{(i)}}{1+\alpha_1}}\Big).
 \end{align*}

The closeness between $\rho_{k,j}^{(1)}$ and $\rho_{k,j}^{(2)}$ leads to the smallness of $I_2$ (see
 (\ref{loc-mass}) (\ref{rho-kj-1}) and (\ref{rho-kj-2})):
\begin{equation}\label{I2}
I_2=\sum_{j=1}^{m}G(p_{k,j}^{(1)},x)(\rho_{k,j}^{(1)}-\rho_{k,j}^{(2)})=O\Big(\sum_{i=1}^{2}e^{-\frac{\lambda_{k,1}^{(i)}}{1+\alpha_1}}\Big).
\end{equation}
For $I_3$, the magnitude of $u_k^{(i)}$ outside the bubbling area determines the smallness of $I_3$:
$$
I_3=\rho_k\int_{M\backslash \bigcup_{j=1}^mB(p_{k,j}^{(1)},\frac{r_0}{2})} G(y,x)H(y)(e^{u_k^{(1)}(y)}-e^{u_k^{(2)}(y)}){\rm d}\mu(y)=O\Big(\sum_{i=1}^{2}e^{-\lambda_{k,1}^{(i)}}\Big).
$$
Therefore
\begin{equation}\label{step2-2}
 u_k^{(1)}-u_k^{(2)}-\big(\bar{u}_k^{(1)}-\bar{u}_k^{(2)}\big)=O\Big(\sum_{i=1}^{2}e^{-\frac{\lambda_{k,1}^{(i)}}{1+\alpha_1}}\Big)\quad \mbox{in}\quad
 M\backslash \bigcup_{j=1}^mB(p_{k,j}^{(1)},r_0).
\end{equation}
To eliminate the averages in (\ref{step2-2}) we take advantage of (\ref{uk-ave-1}) and (\ref{lam-est-1}):
\begin{equation}\label{step2-3}
\bar{u}_k^{(1)}-\bar{u}_k^{(2)}=-(\lambda_{k,j}^{(1)}-\lambda_{k,j}^{(2)})+O\Big(\sum_{i=1}^{2}\lambda_{k,j}^{(i)} e^{-\frac{\lambda_{k,1}^{(i)}}{1+\alpha_1}}\Big)=O\Big(\sum_{i=1}^{2}e^{-\frac{\gamma}{1+\alpha_1}\lambda_{k,1}^{(i)}}\Big),
\end{equation}
Using (\ref{step2-3}) in (\ref{step2-2}) we arrive at
\begin{equation}\label{step2-4}
u_k^{(1)}(x)-u_k^{(2)}(x)=O\Big(\sum_{i=1}^{2}e^{-\frac{\gamma}{1+\alpha_1}\lambda_{k,1}^{(i)}}\Big)=O(e^{-\frac{\gamma}{1+\alpha_1}\lambda_{k,1}^{(1)}}).
\end{equation}
for all $x\in M\backslash \bigcup_{j=1}^mB(p_{k,j}^{(1)},r_0)$.
 Lemma \ref{est1} is established.

\end{proof}

As an immediate application, Lemma \ref{est1} gives ( see (\ref{c}) )
\begin{equation}\label{c-est}
	c_k(x)=e^{u_k^{(1)}(x)}\big(1+O(\parallel u_k^{(1)}-u_k^{(2)}\parallel_{L^{\infty}(M)})\big)=e^{u_k^{(1)}(x)}\big(1+O(e^{-\frac{\gamma}{1+\alpha_1}\lambda_{k,1}^{(1)}})\big).
\end{equation}

To simply the notations, we set
\begin{equation}
	\epsilon_{k,j}=\bigg(\frac{\rho_kh_j(p_{k,j}^{(1)})}{8(1+\alpha_j)^2}\bigg)^{-\frac{1}{2(1+\alpha_j)}}e^{-\frac{\lambda_{k,j}^{(1)}}{2(1+\alpha_j)}}.
\end{equation}
and
\begin{equation}\label{varsigma-kj}
\varsigma_{k,j}(z)=\varsigma_k(\epsilon_{k,j}z+p_{k,j}^{(1)}),\quad |z|<\frac{r_0}{\epsilon_{k,j}},\quad 1\leq j\leq m,
\end{equation}
which satisfies
\begin{equation}\label{sigma-t-equ}
\Delta\varsigma_{k,j}+\frac{8(1+\alpha_j)^2}{\rho_kh_j(p_{k,j}^{(1)})}\rho_k\tilde{h}_j(\epsilon_{k,j}z+p_{k,j}^{(1)})e^{-\lambda_{k,j}^{(1)}}|z|^{2\alpha_j}c_k(\epsilon_{k,j}z+p_{k,j}^{(1)})\varsigma_{k,j}=0.
\end{equation}
for $ |z|<r_0\epsilon_{k,j}^{-1}$.

\medskip

The following lemma determines the limit of $\varsigma_{k,j}$ in both situations:

\begin{lem}\label{lem-limit-1}
{\rm( The limit of} $\varsigma_{k,j}$ {\rm )}

{\rm (i)} For $1\leq j \leq \tau$,
\begin{equation*}
\varsigma_{k,j}\rightarrow b_{j,0}\varphi_{j,0}\quad {\rm in}\ \; C_{loc}(\mathbb{R}^2).
\end{equation*}
where $\varphi_{j,0}$ is
\begin{equation*}
\varphi_{j,0}(z)=\frac{1-|z|^{2(1+\alpha_j)}}{1+|z|^{2(1+\alpha_j)}}.
\end{equation*}

{\rm (ii)} If $\tau<m$ and $\tau+1\leq j \leq m$, there exist constants $ b_{j,0} $, $ b_{j,1} $ and $ b_{j,2} $ such that
\begin{equation*}
\varsigma_{k,j}\rightarrow  b_{j,0}\varphi_{j,0}+b_{j,1}\varphi_{j,1}+b_{j,2}\varphi_{j,2}\quad {\rm in}\ \; C_{loc}(\mathbb{R}^2),
\end{equation*}
where $\varphi_{j,i}$ are
\begin{equation*}
\varphi_{j,0}(z)=\frac{1-|z|^2}{1+|z|^2}, \quad
\varphi_{j,1}(z)=\frac{z_1}{1+|z|^2}, \quad
\varphi_{j,2}(z)=\frac{z_2}{1+|z|^2}.
\end{equation*}
\end{lem}

\begin{proof}[\textbf{Proof}]

(i) For $1\leq j \leq \tau$, it is easy to use (\ref{c-est}) (\ref{eta1}) (\ref{eta2}) and (\ref{eta3}) to obtain
\begin{align*}
&\frac{8(1+\alpha_j)^2}{\rho_kh_j(p_{k,j}^{(1)})}\rho_k\tilde{h}_j(\epsilon_{k,j}z+p_{k,j}^{(1)})e^{-\lambda_{k,j}^{(1)}}c_k(\epsilon_{k,j}z+p_{k,j}^{(1)})  \\
=& \frac{8(1+\alpha_j)^2}{\rho_kh_j(p_{k,j}^{(1)})}\rho_k\tilde{h}_j(\epsilon_{k,j}z+p_{k,j}^{(1)})e^{-\lambda_{k,j}^{(1)}}e^{U_{k,j}^{(1)}+G_{k,j}^{(1)}(\epsilon_{k,j}z+p_{k,j}^{(1)})-G_{k,j}^{(1)}(p_{k,j}^{(1)}) }\big(1+o(1)\big)\\
=& \frac{8(1+\alpha_j)^2}{\big(1+|z|^{2(1+\alpha_j)}\big)^2}\big(1+O(\epsilon_{k,j}|z|)+o(1)\big)
\to \frac{8(1+\alpha_j)^2}{\big(1+|z|^{2(1+\alpha_j)}\big)^2} \quad {\rm in}\ \; C_{loc}(\mathbb{R}^2).
\end{align*}
Therefore, $\varsigma_{k,j}\rightarrow\varsigma_j$ in $C_{loc}(\mathbb{R}^2)$ and $\varsigma_j$ satisfies
\begin{equation}\label{limit-equ1}
\Delta\varsigma_j(z)+\frac{8(1+\alpha_j)^2|z|^{2\alpha_j}}{\big(1+|z|^{2(1+\alpha_j)}\big)^2}\varsigma_j(z)=0 \quad {\rm in} \ \; \mathbb{R}^2.
\end{equation}

Since it is obvious to have $|\varsigma_j|\leq 1$ from $|\varsigma_{k,j}|\leq 1$, we apply Lemma \ref{linear-lem-1} to have $\varsigma_j=b_{j,0}\varphi_{j,0}$ for some constant $b_{j,0}$ and
\begin{equation*}
\varphi_{j,0}(z)=\frac{1-|z|^{2(1+\alpha_j)}}{1+|z|^{2(1+\alpha_j)}}.
\end{equation*}
 That is $\varsigma_{k,j}\rightarrow b_{j,0}\varphi_{j,0}$ in $ C_{loc}(\mathbb{R}^2)$.

\medskip

(ii) For $\tau+1\leq j \leq m$,  by (\ref{c-est}) (\ref{eta1}) and (\ref{eta3}), we have $\varsigma_{k,j}\rightarrow\varsigma_j$ in $C_{loc}(\mathbb{R}^2)$, where
\begin{equation*}
\left\{\begin{array}{lcl}
\Delta\varsigma_j(z)+\frac{8}{(1+|z|^{2})^2}\varsigma_j(z)=0 && {\rm in}\ \; \mathbb{R}^2,
\\
|\varsigma_j|\leq 1  &&{\rm in}\ \; \mathbb{R}^2.
\end{array}
\right.
\end{equation*}
In this case we use Lemma \ref{linear-lem-2} to conclude that
\begin{equation*}
\varsigma_j(z)=b_{j,0}\varphi_{j,0}(z)+b_{j,1}\varphi_{j,1}(z)+b_{j,2}\varphi_{j,2}(z),
\end{equation*}
for some constants $ b_{j,0} $, $ b_{j,1} $ and $ b_{j,2} $. Lemma \ref{lem-limit-1} is established.

\end{proof}

Our next goal is to prove that all $b_{j,0}$ are the same, and equal to the limit of $\varsigma_k$ away from the bubbling area. Our approach is similar to the corresponding parts in
\cite{lin-yan-uniq} for the Chern-Simons-Higgs equation and in \cite{bart-4} for regular mean field equations.

\begin{lem}\label{lem-limit-2}
	
There exists a constant $b_0$ such that
\begin{equation*}
\varsigma_k\rightarrow -b_0\quad {\rm in}\ \; C_{loc}(M\backslash\{p_1,\cdots,p_m\}).
\end{equation*}	
Moreover, $b_{j,0}=b_0$ for all $1 \leq j \leq m$.
\end{lem}

\begin{proof}[\textbf{Proof}]

Starting from the equation for $\varsigma_k$:
\begin{equation*}
\Delta_g\varsigma_k(x)+\rho_k H(x)c_k(x)\varsigma_k(x)=0\quad in\ \; M,
\end{equation*}
we observe from
 (\ref{c-est}) (\ref{sigma1}) and (\ref{sigma2}) that $c_k\rightarrow 0$ in $C_{loc}(M\backslash\{p_1,\cdots,p_m\})$. Since $\|\varsigma_k\|_{L^{\infty}(M)}\leq 1$,  $\varsigma_k\rightarrow \varsigma_0$ in $C_{loc}(M\backslash\{p_1,\cdots,p_m\})$, where $\varsigma_0$ satisfies
\begin{equation}\label{limit-equ2}
\Delta_g\varsigma_0=0\quad {\rm in}\ \; M\backslash\{p_1,\cdots,p_m\}.
\end{equation}
The bound for $\varsigma$: $\|\varsigma_0\|_{L^{\infty}(M)}\leq 1$, which comes from  $\|\varsigma_k\|_{L^{\infty}(M)}\leq 1$, yields the smoothness of $\varsigma_0$ on the whole manifold. Thus $\varsigma_0\equiv -b_0$ in $M$ for some constant $b_0$. In particular,
\begin{equation}\label{limit-equ3}
\varsigma_k\rightarrow -b_0\quad {\rm in}\ \; C_{loc}(M\backslash\{p_1,\cdots,p_m\}).
\end{equation}

For $1\leq j \leq m$, let
\begin{equation*}
	\varphi_{k,j}(x)=\frac{1-\frac{\rho_kh_j(p_{k,j}^{(1)})}{8(1+\alpha_j)^2}e^{\lambda_{k,j}^{(1)}}|x-p_{k,j}^{(1)}|^{2(1+\alpha_j)}}{1+\frac{\rho_kh_j(p_{k,j}^{(1)})}{8(1+\alpha_j)^2}e^{\lambda_{k,j}^{(1)}}|x-p_{k,j}^{(1)}|^{2(1+\alpha_j)}},\quad x\in B(p_{k,j}^{(1)},r_0).
\end{equation*}	
be a sequence of solutions of
\begin{equation*}
-\Delta\varphi_{k,j}(x)=\rho_kh_j(p_{k,j}^{(1)})|x-p_{k,j}^{(1)}|^{2\alpha_j}e^{U_{k,j}^{(1)}}\varphi_{k,j}(x),\quad x\in B(p_{k,j}^{(1)},r_0).
\end{equation*}	
Recall that
\begin{equation*}
-\Delta\varsigma_k(x)=\rho_k \tilde{h}_j(x)|x-p_{k,j}^{(1)}|^{2\alpha_j}\frac{e^{u_k^{(1)}(x)}-e^{u_k^{(2)}(x)}}{u_k^{(1)}(x)-u_k^{(2)}(x)}\varsigma_k(x),\quad x\in B(p_{k,j}^{(1)},r_0).
\end{equation*}	
Using (\ref{p_kj-location}) (\ref{eta1}) (\ref{eta2}), (\ref{eta3}) and integration by parts, we find, for $ d\in (0,r_0) $, that
\begin{align*}
&\int_{\partial B(p_{k,j}^{(1)},d)}\big(\varphi_{k,j}\frac{\partial\varsigma_k}{\partial\nu}-\varsigma_k\frac{\partial\varphi_{k,j}}{\partial\nu}\big){\rm d}\sigma =\int_{B(p_{k,j}^{(1)},d)}\big(\varphi_{k,j}\Delta\varsigma_k-\varsigma_k\Delta\varphi_{k,j}\big){\rm d}x \\
=& \int_{B(p_{k,j}^{(1)},d)}\rho_k\varsigma_k\varphi_{k,j}\Big(-\tilde{h}_j|x-p_{k,j}^{(1)}|^{2\alpha_j}\frac{e^{u_k^{(1)}}-e^{u_k^{(2)}}}{u_k^{(1)}-u_k^{(2)}}+h_j(p_{k,j}^{(1)})|x-p_{k,j}^{(1)}|^{2\alpha_j}e^{U_{k,j}^{(1)}}\Big){\rm d}x \\
=& \int_{B(p_{k,j}^{(1)},d)}\rho_k\varsigma_k\varphi_{k,j}|x-p_{k,j}^{(1)}|^{2\alpha_j}\Big(-\tilde{h}_je^{u_k^{(1)}}\big(1+O(|u_k^{(1)}-u_k^{(2)}|)\big)+h_j(p_{k,j}^{(1)})e^{U_{k,j}^{(1)}}\Big){\rm d}x  \\
=& \int_{B(p_{k,j}^{(1)},d)}\rho_k\varsigma_k\varphi_{k,j}|x-p_{k,j}^{(1)}|^{2\alpha_j}\Big(-\tilde{h}_je^{U_{k,j}^{(1)}+G_{k,j}^{(1)}-G_{k,j}^{(1)}(p_{k,j}^{(1)})+\eta_{k,j}^{(1)}}\big(1+O(|u_k^{(1)}-u_k^{(2)}|)\big)  \\
&\qquad \quad+h_j(p_{k,j}^{(1)})e^{U_{k,j}^{(1)}}\Big){\rm d}x .
\end{align*}

By scaling $x=\epsilon_{k,j}z+p_{k,j}^{(1)}$, (\ref{eta2}), (\ref{eta3}) and the estimate of
	$\parallel u_k^{(1)}-u_k^{(2)}\parallel_{L^{\infty}(M)}$, it is not hard to obtain
\begin{equation}\label{div-est}
	\int_{\partial B(p_{k,j}^{(1)},d)}\big(\varphi_{k,j}\frac{\partial\varsigma_k}{\partial\nu}-\varsigma_k\frac{\partial\varphi_{k,j}}{\partial\nu}\big){\rm d}\sigma=O\big(e^{-\frac{\gamma}{1+\alpha_1}\lambda_{k,1}^{(1)}}\big).
\end{equation}

Let $\varsigma_{k,j}(r)$ be the spherical average of $\varsigma_k$:
$$\varsigma_{k,j}^*(r)=\frac 1{2\pi}\int_{0}^{2\pi}\varsigma_k(r\cos \theta,r\sin \theta){\rm d}\theta, $$ where $r=|x-p_{k,j}^{(1)}|$. Then (\ref{div-est}) yields
\begin{equation*}
(\varsigma_{k,j}^*)'(r)\varphi_{k,j}(r)-\varsigma_{k,j}^*(r)\varphi_{k,j}'(r)=\frac{1}{r}O\big(e^{-\frac{\gamma}{1+\alpha_1}\lambda_{k,1}^{(1)}}\big),\quad r\in(R\epsilon_{k,j},r_0).
\end{equation*}
For any $r \in (R\epsilon_{k,j},r_0)$, we also notice that
\begin{align*}
&\varphi_{k,j}(r)=-1+\frac{1}{r^{2(1+\alpha_j)}}O(e^{-\lambda_{k,j}^{(1)}}),\\
&\varphi_{k,j}'(r)=\frac{1}{r^{2\alpha_j+3}}O(e^{-\lambda_{k,j}^{(1)}}).
\end{align*}
Then we conclude
\begin{equation}\label{ratial-d-est}
(\varsigma_{k,j}^*)'(r)=\frac{1}{r}O\big(e^{-\frac{\gamma}{1+\alpha_1}\lambda_{k,1}^{(1)}}\big)+\frac{1}{r^{2\alpha_j+3}}O(e^{-\lambda_{k,j}^{(1)}}),\quad r\in(R\epsilon_{k,j},r_0).
\end{equation}
Integrating (\ref{ratial-d-est}) from $R\epsilon_{k,j}$ to $r$, we get for all $r\in(R\epsilon_{k,j},r_0)$
\begin{align}\label{ratial-est}
\begin{split}
\varsigma_{k,j}^*(r)
=& \varsigma_{k,j}^*(R\epsilon_{k,j})+O\big(e^{-\frac{\gamma}{1+\alpha_1}\lambda_{k,j}^{(1)}}\big)\big(\log r+\log R+\lambda_{k,j}^{(1)}\big)  \\
&+O\big(e^{-\lambda_{k,j}^{(1)}}\big) \big(r^{-2(1+\alpha_j)}+e^{\lambda_{k,j}^{(1)}}R^{-2(1+\alpha_j)}\big) \\
=& \varsigma_{k,j}^*(R\epsilon_{k,j})+o(1)\log R+O(R^{-2(1+\alpha_j)}) .
\end{split}
\end{align}

The first term of (\ref{ratial-est}) is almost a constant ( Lemma \ref{lem-limit-1} ):
\begin{equation*}
\varsigma_{k,j}^*(R\epsilon_{k,j})=- b_{j,0}+o(1)+o_R(1),
\end{equation*}
where $\lim_{R\to +\infty}o_R(1)=0$. Then it is easy to see from  (\ref{ratial-est}) and (\ref{limit-equ3}) that $b_{j,0}=b_0$ for all $1\le j\le m$.
\end{proof}

Next we introduce a few quantities to be used later. For $1 \leq j\leq m$, let
\begin{align}\label{phi-kj}
\begin{split}
\phi_{k,j}(x)=&\frac{\rho_k}{\sum_{l=1}^m(1+\alpha_l)}\bigg\{(1+\alpha_j)\Big(R(x,p_{k,j}^{(1)})-R(p_{k,j}^{(1)},p_{k,j}^{(1)})\Big) \\
&+\sum_{l\neq j}^{1,\cdots,m}(1+\alpha_l)\Big(G(x,p_{k,l}^{(1)})-G(p_{k,j}^{(1)},p_{k,l}^{(1)})\Big)\bigg\},
\end{split}	
\end{align}
\begin{equation}\label{G-k-tilde}
\tilde{G}_k(x)=8\pi\sum_{l=1}^{m}(1+\alpha_l)G(x,p_{k,l}^{(1)}),
\end{equation}
It is easy to see that in $M\setminus \{p_{k,1}^{(1)},\cdots,p_{k,m}^{(1)}\}$,
\begin{equation}\label{G-phi-equ}
\nabla\big(\tilde{G}_k(x)-\phi_{k,j}(x)\big)=-4(1+\alpha_j)\frac{x-p_{k,j}^{(1)}}{|x-p_{k,j}^{(1)}|^2},\quad \Delta\big(\tilde{G}_k(x)-\phi_{k,j}(x)\big)=0.
\end{equation}
Set
\begin{equation}\label{v-kj}
v_{k,j}^{(i)}(x)=u_k^{(i)}(x)-\phi_{k,j}(x),\quad  i=1,2,
\end{equation}
and
\begin{equation}\label{A-kj}
A_{k,j}=\int_{M_j} f_k \, {\rm d}\mu,
\end{equation}
then we estimate $\nabla v_{k,j}^{(i)}$ away from the bubbling area:

\begin{lem}\label{lem-Dv-kj}
	
	For any $\theta\in(0,r_0)$ small enough and $x\in B(p_{k,j}^{(1)},2r_0)\setminus B(p_{k,j}^{(1)},\theta)$, the gradient of $v_{k,j}^{(i)}$ is very close to that of a harmonic function:
	\begin{align}\label{Dv-kj-est}
	\nabla v_{k,j}^{(i)}(x)=-4(1+\alpha_j)\frac{x-p_{k,j}^{(1)}}{|x-p_{k,j}^{(1)}|^2}+O(e^{-\frac{\lambda_{k,j}^{(1)}}{1+\alpha_1}}).
	\end{align}
\end{lem}

\begin{proof}[\textbf{Proof}]
	
	By (\ref{sigma1}) (\ref{sigma2}) and Theorem \ref{Theorem chen-lin2}, we have
	\begin{align*}
	O(e^{-\frac{\lambda_{k,j}^{(1)}}{1+\alpha_1}})=\nabla \sigma_k^{(i)}=\nabla (v_{k,j}^{(i)}+\phi_{k,j}-\tilde{G}_k)+O(e^{-\frac{\lambda_{k,j}^{(1)}}{1+\alpha_1}}), \quad x\in B(p_{k,j}^{(1)},2r_0)\setminus B(p_{k,j}^{(1)},\theta)
	\end{align*}
	 for $i=1,2$, $1\leq j\leq m$. Consequently using (\ref{phi-kj})$\sim$(\ref{G-phi-equ}) we have
	\begin{align*}
	\nabla v_{k,j}^{(i)}(x)=&\nabla \big(\tilde{G}_k(x)-\phi_{k,j}(x)\big)+O(e^{-\frac{\lambda_{k,j}^{(1)}}{1+\alpha_1}})\\
	=&-4(1+\alpha_j)\frac{x-p_{k,j}^{(1)}}{|x-p_{k,j}^{(1)}|^2}+O(e^{-\frac{\lambda_{k,j}^{(1)}}{1+\alpha_1}}).
	\end{align*}
	
\end{proof}
Next we estimate $\varsigma_k$ and its derivatives away from blowup points.
\begin{lem}\label{lem-C1-est}
	
	Given $\theta\in(0,r_0)$ small enough, we have
	\begin{align}\label{GRF-est-1}
	\begin{split}
	\varsigma_k-\bar{\varsigma}_k=\sum_{j=1}^m A_{k,j}G(p_{k,j}^{(1)},x)+o(e^{-\frac{\lambda_{k,1}^{(1)}}{2(1+\alpha_1)}}) \quad {\rm in}\ \; M\setminus\bigcup_{j=1}^m B(p_{k,j}^{(1)},\theta).
	\end{split}
	\end{align}
	
\end{lem}

\begin{proof}[\textbf{Proof}]
	From the Green's representation formula for $u_k^{(i)}$ and the definition of $\varsigma_k$, we have the following expression of $\varsigma_k$:
$$\varsigma_k(x)-\bar \varsigma_k=\int_MG(y,x)f_k(y)d\mu(y). $$
Then for $ x\in M\setminus\bigcup_{j=1}^m B(p_{k,j}^{(1)},\theta)$ we evaluate the integral in three parts:
	\begin{align}\label{J1+J2+J3}
	\begin{split}
	\varsigma_k(x)-\bar{\varsigma}_k=&\sum_{j=1}^m A_{k,j}G(p_{k,j}^{(1)},x)+\sum_{j=1}^\tau\int_{M_j}\big(G(y,x)-G(p_{k,j}^{(1)},x)\big)f_k(y){\rm d}\mu(y) \\
	&+\sum_{j=\tau+1}^m\int_{M_j}\big(G(y,x)-G(p_{k,j}^{(1)},x)\big)f_k(y){\rm d}\mu(y) \\
	=&\mathrel{\mathop:}J_1+J_2+J_3
	\end{split}
	\end{align}
	Note that $J_3=\emptyset$ if  $\tau=m$.
	Then it follows from the definition of $\eta_{k,j}^{(1)}$, (\ref{eta2}) and (\ref{eta3}) that
	\begin{align*}
	&\int_{M_j}\langle\partial_yG(y,x)\big|_{y=p_{k,j}^{(1)}},y-p_{k,j}^{(1)}\rangle f_k(y){\rm d}\mu(y)  \\
	=&\int_{B(p_{k,j}^{(1)},r_0)} \langle\partial_yG(y,x)\big|_{y=p_{k,j}^{(1)}},y-p_{k,j}^{(1)}\rangle f_k(y)e^{\phi_j(y)}{\rm d}y+O(e^{-\lambda_{k,j}^{(1)}}) \\
	=&\int_{B(p_{k,j}^{(1)},r_0)} \langle\partial_yG(y,x)\big|_{y=p_{k,j}^{(1)}},y-p_{k,j}^{(1)}\rangle \rho_kh_j(p_{k,j}^{(1)})|y-p_{k,j}^{(1)}|^{2\alpha_j}e^{U_{k,j}^{(1)}} \varsigma_k(y)  \\
	&\times\big(1+O(|y-p_{k,j}^{(1)}|+\epsilon_{k,j}+\epsilon_{k,1}^2)\big)(1+o(1)){\rm d}y+O(e^{-\lambda_{k,j}^{(1)}}).
	\end{align*}
	
	For $1\leq j\leq \tau$, using $y=\epsilon_{k,j}z+p_{k,j}^{(1)}$ in the evaluation of the identity above, we have
	\begin{align*}
	&\int_{B(p_{k,j}^{(1)},r_0)} \langle\partial_yG(y,x)\big|_{y=p_{k,j}^{(1)}},y-p_{k,j}^{(1)}\rangle f_k(y)e^{\phi_j}{\rm d}y  \\
	=&\epsilon_{k,j}\int_{|z|<\frac{r_0}{\epsilon_{k,j}}}\langle\partial_yG(y,x)\big|_{y=p_{k,j}^{(1)}},z\rangle\frac{8(1+\alpha_j)^2|z|^{2\alpha_j}}{(1+|z|^{2(1+\alpha_j)})^2}\Big(b_0\frac{1-|z|^{2(1-\alpha_j)}}{1+|z|^{2(1+\alpha_j)}}+o(1)\Big)  \\
	&\times\big(1+O(\epsilon_{k,j}|z|+\epsilon_{k,j}+\epsilon_{k,1}^2)+o(1)\big){\rm d}z.	 \\
	=&o(e^{-\frac{\lambda_{k,j}^{(1)}}{2(1+\alpha_j)}}),
	\end{align*}

	If $\tau <m$, let us recall that $\alpha_j=0$ for $\tau+1\leq j\leq m$. Similarly, by the standard scaling, Lemma \ref{lem-limit-1} and symmetry, we have
	\begin{align*}
	&\int_{B(p_{k,j}^{(1)},r_0)} \langle\partial_yG(y,x)\big|_{y=p_{k,j}^{(1)}},y-p_{k,j}^{(1)}\rangle  f_k(y)e^{\phi_j}{\rm d}y  \\
	=&\epsilon_{k,j}\int_{|z|<\frac{r_0}{\epsilon_{k,j}}}\langle\partial_yG(y,x)\big|_{y=p_{k,j}^{(1)}},z\rangle\frac{8}{(1+|z|^2)^{2}}\varsigma_{k,j}(z){\rm d}z+o(e^{-\frac{\lambda_{k,j}^{(1)}}{2}})	 \\
	=&e^{-\frac{\lambda_{k,j}^{(1)}}{2}}\Big(\sum_{h=1}^2\partial_{y_h}G(y,x)\big|_{y=p_{k,j}^{(1)}}b_{j,h}\Big)B_j +o(e^{-\frac{\lambda_{k,j}^{(1)}}{2}}),
	\end{align*}
	where
	\begin{equation}\label{Bj}
		B_j=4\sqrt{\frac{8}{\rho_kh_j(p_{k,j}^{(1)})}}\displaystyle\int_{\mathbb{R}^2}\frac{|z|^2}{(1+|z|^2)^3}{\rm d}z .
	\end{equation}
	For the second order terms in the expansion of $G$, we have
	\begin{align*}
	&\int_{M_j}|y-p_{k,j}^{(1)}|^2f_k{\rm d}\mu(y) \\
	=&\int_{B(p_{k,j}^{(1)},r_0)}|y-p_{k,j}^{(1)}|^2f_ke^{\phi_j}{\rm d}y+O(e^{-\lambda_{k,j}^{(1)}})   \\
	=&O(e^{-\frac{\lambda_{k,j}^{(1)}}{1+\alpha_j}})\int_{|z|<\frac{r_0}{\epsilon_{k,j}}}\frac{|z|^{2(1+\alpha_j)}}{(1+|z|^{2(1+\alpha_j)})^2}(1+o(1)){\rm d}z+O(e^{-\lambda_{k,j}^{(1)}})  \\
	=&\left\{
	\begin{array}{lcl}
	O(e^{-\frac{1}{1+\alpha_j}\lambda_{k,j}^{(1)}}), && 1\leq j\leq \tau, \\
	O(\lambda_{k,j}^{(1)}e^{-\lambda_{k,j}^{(1)}}),  && \tau+1\leq j\leq m.
	\end{array}
	\right.
	\end{align*}
	
	Consequently for $J_2$ and $J_3$ we have
	\begin{equation}\label{J2}
		J_2=o(e^{-\frac{\lambda_{k,1}^{(1)}}{2(1+\alpha_1)}}),
	\end{equation}
	\begin{equation}\label{J3}
		J_3=\sum_{j=\tau+1}^me^{-\frac{\lambda_{k,j}^{(1)}}{2}}\Big(\sum_{h=1}^2\partial_{y_h}G(y,x)\big|_{y=p_{k,j}^{(1)}}b_{j,h}\Big)B_j +o(e^{-\frac{\lambda_{k,1}^{(1)}}{2}}).
	\end{equation}
	
  Observing (\ref{J1+J2+J3}) (\ref{J2}) and (\ref{J3}), we conclude that (\ref{GRF-est-1}) holds. Then by standard estimates we also have
  \begin{align}\label{Dsigma_k-1}
  \nabla\varsigma_k(x)=\sum_{j=\tau+1}^m A_{k,j}\nabla_xG(p_{k,j}^{(1)},x)+o(e^{-\frac{\lambda_{k,1}^{(1)}}{2(1+\alpha_1)}}), \quad x\in M\setminus\bigcup_{j=1}^m B(p_{k,j}^{(1)},\theta).
  \end{align}
	
\end{proof}

\section{Estimates associated with Pohozaev identities}\label{anal-pohozaev}

 In this section, we establish some sharp estimates for certain terms crucial for evaluation of Pohozaev identities.

\smallskip

The first important quantity is $A_{k,j}$, defined in (\ref{GRF-est-1}) and the study of which is through the following Pohozaev identity:

\begin{lem}\label{Pohozaev identity-1}

	For $1 \leq j \leq m$ and any $r\in(0,r_0)$, it holds that
	\begin{align}\label{PI-1}
	\begin{split}
	& \frac{1}{2}\int_{\partial B(p_{k,j}^{(1)},r)}r\langle \nabla v_{k,j}^{(1)}+\nabla v_{k,j}^{(2)},\nabla \varsigma_k\rangle{\rm d}\sigma \\
	&  -\int_{\partial B(p_{k,j}^{(1)},r)}r\langle\nu,\nabla v_{k,j}^{(1)}+\nabla v_{k,j}^{(2)}\rangle  \langle\nu,\nabla\varsigma_k\rangle {\rm d}\sigma \\
	=& \int_{\partial B(p_{k,j}^{(1)},r)}\frac{r\rho_k\tilde{h}_j|x-p_{k,j}^{(1)}|^{2\alpha_j}\big(e^{v_{k,j}^{(1)}+\phi_{k,j}}-e^{v_{k,j}^{(2)}+\phi_{k,j}}\big)}{\parallel v_{k,j}^{(1)}-v_{k,j}^{(2)} \parallel_{L^{\infty}(M)} } {\rm d}\sigma \\
	&  - 2(1+\alpha_j)\int_{B(p_{k,j}^{(1)},r)} \frac{\rho_k\tilde{h}_j|x-p_{k,j}^{(1)}|^{2\alpha_j}\big(e^{v_{k,j}^{(1)}+\phi_{k,j}}-e^{v_{k,j}^{(2)}+\phi_{k,j}}\big)}{\parallel v_{k,j}^{(1)}-v_{k,j}^{(2)} \parallel_{L^{\infty}(M)}} {\rm d}x\\
	&  -\int_{B(p_{k,j}^{(1)},r)} \frac{\rho_k\tilde{h}_j|x-p_{k,j}^{(1)}|^{2\alpha_j}\big(e^{v_{k,j}^{(1)}+\phi_{k,j}}-e^{v_{k,j}^{(2)}+\phi_{k,j}}\big)}{\parallel v_{k,j}^{(1)}-v_{k,j}^{(2)} \parallel_{L^{\infty}(M)}} \langle \nabla\big(\log \tilde{h}_j+\phi_{k,j}\big),x-p_{k,j}^{(1)}\rangle {\rm d}x.
	\end{split}
	\end{align}
	
\end{lem}
This Pohozaev identity has been used in \cite{bart-4} and \cite{lin-yan-uniq}, we include the proof for the convenience of the readers.
\begin{proof}[\textbf{Proof}]
	
	First we observe that for any two smooth functions $u$ and $v$,
	\begin{align}\label{divergence}
	\begin{split}
	&\Delta u\{\nabla v\cdotp (x-p_{k,j}^{(1)})\}+\Delta u\{\nabla v\cdotp (x-p_{k,j}^{(1)})\} \\
	=\,&{\rm div}\Big\{\nabla u\big[\nabla v\cdotp(x-p_{k,j}^{(1)})\big]+\nabla v\big[\nabla u\cdotp(x-p_{k,j}^{(1)})\big] - \nabla u\cdotp\nabla v (x-p_{k,j}^{(1)})\Big\}.
	\end{split}
	\end{align}
	
Replacing $u$, $v$ by $v_{k,j}^{(1)}- v_{k,j}^{(2)}$ and $v_{k,j}^{(1)}+ v_{k,j}^{(2)}$ respectively in (\ref{divergence}), we have
	\allowdisplaybreaks
	\begin{align}\label{div-formula-1}
	\begin{split}
	&\Delta (v_{k,j}^{(1)}- v_{k,j}^{(2)})\big\{\nabla (v_{k,j}^{(1)}+ v_{k,j}^{(2)})\cdotp (x-p_{k,j}^{(1)})\big\} +\Delta (v_{k,j}^{(1)}+ v_{k,j}^{(2)})\big\{\nabla (v_{k,j}^{(1)}- v_{k,j}^{(2)})\cdotp (x-p_{k,j}^{(1)})\big\} \\
	=&{\rm div}\Big\{\nabla (v_{k,j}^{(1)}- v_{k,j}^{(2)})\big[\nabla (v_{k,j}^{(1)}+ v_{k,j}^{(2)})\cdotp(x-p_{k,j}^{(1)})\big] \\
	+&\nabla (v_{k,j}^{(1)}+ v_{k,j}^{(2)})\big[\nabla (v_{k,j}^{(1)}- v_{k,j}^{(2)})\cdotp(x-p_{k,j}^{(1)})\big]-\nabla (v_{k,j}^{(1)}- v_{k,j}^{(2)})\cdotp\nabla (v_{k,j}^{(1)}+ v_{k,j}^{(2)})(x-p_{k,j}^{(1)})\Big\}.
	\end{split}
	\end{align}
	By the definition of $v_{k,j}^{(i)}$, we see that, for $x\in B(p_{k,j}^{(1)},r_0)$,
	\begin{align}\label{v-kj-equ}
	\Delta(v_{k,j}^{(1)}\pm v_{k,j}^{(2)})+\rho_ke^{\phi_{k,j}}\tilde{h}_j|x-p_{k,j}^{(1)}|^{2\alpha_j}(e^{u_k^{(1)}}\pm e^{u_k^{(2)}})=0.
	\end{align}
	Using (\ref{v-kj-equ}) and
 $$g_i=e^{v_{k,j}^{(i)}+\phi_{k,j}+\log \tilde{h}_j+2\alpha_j\log|x-p_{k,j}^{(1)}|}$$
 the right hand side (RHS) of (\ref{div-formula-1}) can be written as:
	\begin{align*}
	&{\rm (RHS)}\ {\rm of}\ (\ref{div-formula-1})\\
	=&-\rho_k(g_1-g_2)\big\{\nabla (v_{k,j}^{(1)}+ v_{k,j}^{(2)})\cdotp (x-p_{k,j}^{(1)})\big\} -\rho_k(g_1+g_2)\big\{\nabla (v_{k,j}^{(1)}- v_{k,j}^{(2)})\cdotp (x-p_{k,j}^{(1)})\big\}\\
	=&-2\rho_kg_1\big\{\nabla v_{k,j}^{(1)}\cdotp (x-p_{k,j}^{(1)})\big\}+2\rho_kg_2\big\{\nabla v_{k,j}^{(2)}\cdotp (x-p_{k,j}^{(1)})\big\}\\
	=&-{\rm div}\big\{2\rho_k(g_1+g_2) (x-p_{k,j}^{(1)})\big\}+2\rho_k(g_1-g_2)\big\{\nabla (\log \tilde{h}_j+\phi_{k,j})\cdotp (x-p_{k,j}^{(1)})\big\}\\
	&+2\rho_k(g_1-g_2)\big\{2\alpha_j\nabla \log |x-p_{k,j}^{(1)}|\cdotp (x-p_{k,j}^{(1)})\big\}+4\rho_k(g_1-g_2)\\
	=&-{\rm div}\big\{2\rho_k(g_1+g_2) (x-p_{k,j}^{(1)})\big\}+4(1+\alpha_j)\rho_k(g_1-g_2)\\
	&+2\rho_k(g_1-g_2)\big\{\nabla (\log \tilde{h}_j+\phi_{k,j})\cdotp (x-p_{k,j}^{(1)})\big\}.
	\end{align*}
	 Since $\varsigma_k=\frac{v_{k,j}^{(1)} -v_{k,j}^{(2)}}{\parallel v_{k,j}^{(1)}-v_{k,j}^{(2)} \parallel_{L^{\infty}(M)}}$ and $\nu=\frac{x-p_{k,j}^{(1)}}{r}$, we have
	\begin{align}\label{div-formula-2}
	\begin{split}
	&\int_{\partial B(p_{k,j}^{(1)},r)}\frac{{\rm (RHS)}\ {\rm of}\  (\ref{div-formula-1})}{\parallel v_{k,j}^{(1)}-v_{k,j}^{(2)} \parallel_{L^{\infty}(M)}}{\rm d}\sigma\\
	=&-2\int_{\partial B(p_{k,j}^{(1)},r)}\frac{r\rho_k\tilde{h}_j|x-p_{k,j}^{(1)}|^{2\alpha_j}\big(e^{v_{k,j}^{(1)}+\phi_{k,j}}+e^{v_{k,j}^{(2)}+\phi_{k,j}}\big)}{\parallel v_{k,j}^{(1)}-v_{k,j}^{(2)} \parallel_{L^{\infty}(M)} } {\rm d}\sigma\\
	&  + 4(1+\alpha_j)\int_{B(p_{k,j}^{(1)},r)} \frac{\rho_k\tilde{h}_j|x-p_{k,j}^{(1)}|^{2\alpha_j}\big(e^{v_{k,j}^{(1)}+\phi_{k,j}}-e^{v_{k,j}^{(2)}+\phi_{k,j}}\big)}{\parallel v_{k,j}^{(1)}-v_{k,j}^{(2)} \parallel_{L^{\infty}(M)}} {\rm d}x\\
	&  +2\int_{B(p_{k,j}^{(1)},r)} \frac{\rho_k\tilde{h}_j|x-p_{k,j}^{(1)}|^{2\alpha_j}\big(e^{v_{k,j}^{(1)}+\phi_{k,j}}-e^{v_{k,j}^{(2)}+\phi_{k,j}}\big)}{\parallel v_{k,j}^{(1)}-v_{k,j}^{(2)} \parallel_{L^{\infty}(M)}} \langle \nabla\big(\log \tilde{h}_j+\phi_{k,j}\big),x-p_{k,j}^{(1)}\rangle {\rm d}x.
	\end{split}
	\end{align}
	On the other hand,
	\begin{align}\label{div-formula-3}
	\begin{split}
	&\int_{\partial B(p_{k,j}^{(1)},r)}\frac{{\rm (LHS)}\ {\rm of}\  (\ref{div-formula-1})}{\parallel v_{k,j}^{(1)}-v_{k,j}^{(2)} \parallel_{L^{\infty}(M)}}{\rm d}\sigma\\
	=&-\int_{\partial B(p_{k,j}^{(1)},r)}r\langle \nabla v_{k,j}^{(1)}+\nabla v_{k,j}^{(2)},\nabla \varsigma_k\rangle{\rm d}\sigma \\
	&+2\int_{\partial B(p_{k,j}^{(1)},r)}r\langle\nu,\nabla v_{k,j}^{(1)}+\nabla v_{k,j}^{(2)}\rangle \langle\nu,\nabla \varsigma_k\rangle {\rm d}\sigma
	\end{split}
	\end{align}
	Then (\ref{PI-1}) follows from (\ref{div-formula-1}), (\ref{div-formula-2}) and (\ref{div-formula-3}).
	
\end{proof}

\begin{rem}
	It is easy to see Pohozaev-type identity (\ref{PI-1}) also holds for $\alpha_j>-1$.
\end{rem}

\begin{lem}\label{lem-PI1-left}

For all $1\leq j\leq m$,
	\begin{equation}\label{PI-1-l}
	{\rm (LHS)}\ {\rm of}\  (\ref{PI-1})=-4(1+\alpha_j)A_{k,j}+O(e^{-\frac{\lambda_{k,j}^{(1)}}{1+\alpha_1}}\sum_{l=1}^{m}|A_{k,l}|)+o(e^{-\frac{\lambda_{k,j}^{(1)}}{1+\alpha_1}}).
	\end{equation}
\end{lem}

\begin{proof}[\textbf{Proof}]
	
	From (\ref{Dv-kj-est}) and (\ref{Dsigma_k-1}), we find that
	\begin{align*}
	&{\rm (LHS)}\ {\rm of}\  (\ref{PI-1})\\
	=&\,4(1+\alpha_j)\int_{\partial B(p_{k,j}^{(1)},r)}\langle\nu,D\varsigma_k\rangle{\rm d}\sigma+O(e^{-\frac{\lambda_{k,j}^{(1)}}{1+\alpha_1}}\parallel D\varsigma_k\parallel_{L^{\infty}(\partial B(p_{k,j}^{(1)},r))})\\
	=&\,4(1+\alpha_j)\int_{\partial B(p_{k,j}^{(1)},r)}\langle\nu,D\varsigma_k\rangle{\rm d}\sigma+O(e^{-\frac{\lambda_{k,j}^{(1)}}{1+\alpha_1}}\sum_{l=1}^{m}|A_{k,l}|)+o(e^{-\frac{3}{2(1+\alpha_1)}\lambda_{k,j}^{(1)}}).
	\end{align*}
	For  $x\in\partial B(p_{k,j}^{(1)},r)$ we use the Green's representation formula to estimate $\varsigma_k(x)$:
	\begin{align*}
	&\varsigma_k(x)-\bar{\varsigma}_k=\int_{M}G(y,x)f_k(y){\rm d}\mu(y)\\ =&\sum_{l=1}^mA_{k,l}G(p_{k,l}^{(1)},x)+\sum_{l=1}^m\sum_{h=1}^2B_{k,l,h}\partial_{y_h}G(y,x)\big|_{y=p_{k,l}^{(1)}}+\frac{1}{2}\sum_{l=1}^m\sum_{h,i=1}^2C_{k,l,h,i}\partial_{y_hy_i}^2G(y,x)\big|_{y=p_{k,l}^{(1)}}\\
	&+O(1)\sum_{l=1}^m\int_{M_l}|y-p_{k,j}^{(1)}|^3f_k{\rm d}\mu(y)  \quad {\rm in}\,\ C^1\big(B(p_{k,j}^{(1)},2r_0)\setminus B(p_{k,j}^{(1)},\frac{r}{2})\big),
	\end{align*}
	where
	\begin{align*}
	&	B_{k,l,h}=\int_{M_l}(y-p_{k,l}^{(1)})_hf_k(y){\rm d}\mu(y),
	\\
	& C_{k,l,h,i}=\int_{M_l}(y-p_{k,l}^{(1)})_h(y-p_{k,l}^{(1)})_if_k(y){\rm d}\mu(y).
	\end{align*}
	
	It is easy to see that the last term is rather minor:
	\begin{align*}
	&\sum_{l=1}^m\int_{M_l}|y-p_{k,j}^{(1)}|^3f_k(y){\rm d}\mu(y)   \\ =&\sum_{l=1}^m\int_{B(p_{k,l}^{(1)},r)}\frac{e^{\lambda_{k,l}^{(1)}}|y-p_{k,l}^{(1)}|^{2\alpha_l+3}}{\big(1+e^{\lambda_{k,l}^{(1)}}|y-p_{k,l}^{(1)}|^{2(1+\alpha_l)}\big)^2}{\rm d}y +O(e^{-\lambda_{k,1}^{(1)}})  \\
	=&\sum_{l=1}^mO(e^{-\frac{3}{2(1+\alpha_l)}\lambda_{k,l}^{(1)}})\int_{|z|<\frac{r}{\epsilon_{k,l}}}\frac{|z|^{2\alpha_l+3}}{(1+|z|^{2(1+\alpha_l)})^2}{\rm d}z+O(e^{-\lambda_{k,1}^{(1)}})  \\
	=&o(e^{-\frac{\lambda_{k,1}^{(1)}}{1+\alpha_1}}).
	\end{align*}
Setting
	\begin{align*}
	\bar{G}_k(x)=&\bar{\varsigma}_k(x)+\sum_{j=1}^m A_{k,j}G(p_{k,j}^{(1)},x)+\sum_{l=1}^m\sum_{h=1}^2B_{k,l,h}\partial_{y_h}G(y,x)\big|_{y=p_{k,l}^{(1)}}\\
	&+\frac{1}{2}\sum_{l=1}^m\sum_{h,i=1}^2C_{k,l,h,i}\partial_{y_hy_i}^2G(y,x)\big|_{y=p_{k,l}^{(1)}},
	\end{align*}
we now have
	\begin{equation*}
		\nabla\varsigma_k(x)-\nabla\bar{G}_k(x)=o(e^{-\frac{\lambda_{k,1}^{(1)}}{1+\alpha_1}}).
	\end{equation*}
	Thus
	\begin{align}\label{pi-l-1}
	\begin{split}
	&{\rm (LHS)}\ {\rm of}\  (\ref{PI-1})\\
	=&\,4(1+\alpha_j)\int_{\partial B(p_{k,j}^{(1)},r)}\langle\nu,\nabla\bar{G}_k\rangle{\rm d}\sigma+O(e^{-\frac{\lambda_{k,j}^{(1)}}{1+\alpha_1}}\sum_{l=1}^{m}|A_{k,l}|)+o(e^{-\frac{\lambda_{k,j}^{(1)}}{1+\alpha_1}}).
	\end{split}
	\end{align}
	Now we take the global cancellation property into consideration: for any fixed $\theta\in(0,r)$,
	\begin{equation}\label{Gbar-equ}
	\Delta \bar{G}_k=\sum_{l=1}^{m}A_{k,l}=\int_{M}f_k\,{\rm d}\mu=0,\quad {\rm in}\ \;B(p_{k,j}^{(1)},2r_0)\setminus B(p_{k,j}^{(1)},\theta).
	\end{equation}
	Using (\ref{G-phi-equ}) (\ref{Gbar-equ}) and (\ref{divergence}), we have
	\begin{align*}
	0=&\int_{B_r\setminus B_{\theta}} \Big\{\Delta\bar{G}_k\big\{\nabla(\tilde{G}_k-\phi_{k,j})\cdotp (x-p_{k,j}^{(1)})\big\}+\Delta(\tilde{G}_k-\phi_{k,j})\big\{\nabla \bar{G}_k\cdotp (x-p_{k,j}^{(1)})\big\}\Big\}{\rm d}x\\
	=&-4\pi(1+\alpha_j)\int_{\partial(B_r\setminus B_{\theta})}\frac{\partial\bar{G}_k}{\partial\nu} \,{\rm d}\sigma,
	\end{align*}
	where $B(p_{k,j}^{(1)},r)$ and $B(p_{k,j}^{(1)},\theta)$ are replaced by $B_r$, $B_{\theta}$  respectively for simplicity. Therefore
	\begin{equation}\label{pi-l-2}
	\int_{\partial B_r}\frac{\partial\bar{G}_k}{\partial\nu} {\rm d}\sigma=\int_{\partial B_{\theta}}\frac{\partial\bar{G}_k}{\partial\nu} {\rm d}\sigma.
	\end{equation}
	Further direct computation yields
	\begin{align}\label{pi-l-3}
	\begin{split}
	&\int_{\partial B_{\theta}}\langle\nu,\sum_{l=1}^{m}A_{k,l}\nabla_xG(p_{k,l}^{(1)},x)\rangle{\rm d}\sigma\\
	=& -A_{k,j}\int_{\partial B_{\theta}}\langle\nu,\nabla_xG(p_{k,l}^{(1)},x)\rangle{\rm d}\sigma+o_{\theta}(1)\\
	=& -A_{k,j}\int_{\partial B_{\theta}}\langle\nu,\nabla_x\frac{1}{2\pi}\log |x-p_{k,l}^{(1)}|\rangle{\rm d}\sigma+o_{\theta}(1)  \\
	=&-A_{k,j}+o_{\theta}(1),
	\end{split}
	\end{align}
	where $\lim_{\theta\to 0}o_{\theta}(1)=0$, and we have used the fact that all the terms related to $l\neq j$ are minor. Let us observe that
	\begin{equation}\label{pi-l-4}
	\begin{split}
	&\int_{\partial B(0,\theta)}\langle\nu,\nabla_x\partial_{y_h}\log|z|\rangle{\rm d}\sigma=-\int_{\partial B(0,\theta)}\sum_{i=1}^2\frac{z_i}{|z|}\frac{\delta_{ih}|z|^2-2z_iz_h}{z^4}{\rm d}\sigma=0, \\
	&\int_{\partial B(0,\theta)}\langle\nu,\nabla_x\frac{\partial^2}{\partial y_h^2}\log|z|\rangle{\rm d}\sigma=-\int_{\partial B(0,\theta)}\big(\frac{2}{|z|^3}-\frac{4z_h^2}{|z|^5}\big){\rm d}\sigma=0, \\
	&\int_{\partial B(0,\theta)}\langle\nu,\nabla_x\frac{\partial^2}{\partial y_hy_i}\log|z|\rangle{\rm d}\sigma=-\int_{\partial B(0,\theta)}\big(\frac{4z_hz_i}{|z|^5}-\frac{8z_hz_i}{|z|^5}\big){\rm d}\sigma=0,
	\end{split}
	\end{equation}
	
	Obviously, from (\ref{pi-l-2})$\sim$(\ref{pi-l-4}) we can see that
	\begin{equation*}
		\int_{\partial B(p_{k,j}^{(1)},r)}\langle \nu,\nabla\bar{G}_k\rangle{\rm d}\sigma=-A_{k,j}+o_{\theta}(1).
	\end{equation*}
	which together with (\ref{pi-l-1}), concludes the proof of Lemma \ref{lem-PI1-left}.
	
\end{proof}

\begin{lem}\label{lem-PI1-right}
	\begin{equation}\label{PI-1-r-1}
	\begin{split}
	&{\rm (RHS)}\ {\rm of}\  (\ref{PI-1})\\
	=&-2(1+\alpha_j)A_{k,j}-\frac{4\pi^2\big[\Delta \log h(p_j)+\rho_*-N^*-2K(p_j)\big]}{\big(\rho_kh_j(p_{k,j}^{(1)})\big)^{\frac{1}{1+\alpha_1}}(1+\alpha_1)\sin \frac{\pi}{1+\alpha_1}}b_0e^{-\frac{\lambda_{k,j}^{(1)}}{1+\alpha_1}}\\
	&+o(e^{-\frac{\lambda_{k,j}^{(1)}}{1+\alpha_1}}),\qquad 1\leq j\leq t.
	\end{split}
	\end{equation}
	\begin{align}
	&{\rm (RHS)}\ {\rm of}\  (\ref{PI-1})=-2(1+\alpha_j)A_{k,j}+o(e^{-\frac{\lambda_{k,j}^{(1)}}{2(1+\alpha_1)}}),\quad t+1\leq j\leq \tau. \label{PI-1-r-2} \\
	&{\rm (RHS)}\ {\rm of}\  (\ref{PI-1})=-2(1+\alpha_j)A_{k,j}+o(e^{-\frac{\lambda_{k,j}^{(1)}}{1+\alpha_1}}),\quad \tau+1\leq j\leq m.    \label{PI-1-r-3}
	\end{align}
\end{lem}

\begin{proof}[\textbf{Proof}]
	
	We use $K_1,K_2,K_3$ to denote the three terms on the right hand of (\ref{PI-1}). The first two terms are quite easy to estimate:
	\begin{equation}\label{K1}
	K_1=\int_{\partial B(p_{k,j}^{(1)},r)}rf_k\,{\rm d}\sigma=\int_{\partial B(p_{k,j}^{(1)},r)}\rho_k\tilde{h}_jr^{2\alpha_j+1}e^{u_k^{(1)}}(\varsigma_k+o(1)){\rm d}\sigma=O(e^{-\lambda_{k,j}^{(1)}}),
	\end{equation}
	\begin{equation}\label{K2}
	K_2=-2(1+\alpha_j)A_{k,j}.
	\end{equation}
	More work is needed for
	\begin{equation*}
	K_3=-\int_{B(p_{k,j}^{(1)},r)}f_ke^{\phi_j}\langle \nabla(\log \tilde{h}_j+\phi_{k,j}),x-p_{k,j}^{(1)}\rangle{\rm d}x.
	\end{equation*}
	First we use $\nabla\phi_j(p_{k,j}^{(1)})=0$ to write $\nabla (\log \tilde h_j+\phi_{k,j})(x)$ as
	\begin{equation}\label{expansion}
	\begin{split}
	&\nabla(\log \tilde{h}_j+\phi_{k,j})(x)\\
	=&\nabla(\log h_j+\phi_{k,j})(p_{k,j}^{(1)})+\langle D^2(\log \tilde{h}_j+\phi_{k,j})(p_{k,j}^{(1)}),x-p_{k,j}^{(1)}\rangle+O(|x-p_{k,j}^{(1)}|^2).
	\end{split}
	\end{equation}
	Then we evaluate $K_3$ in three cases:

	$\mathbf{Case\ 1:}$ $1\leq j\leq t$ $(\alpha_j=\alpha_1)$. The assumption $D(\mathbf{p})=0$ and (\ref{phi-kj}) (\ref{rho-k}) imply
	\begin{equation*}
		\nabla(\log h_j+\phi_{k,j})(p_{k,j}^{(1)})=O(e^{-\frac{\lambda_{k,j}^{(1)}}{1+\alpha_1}}).
	\end{equation*}
	Thus, after the scaling $x=\epsilon_{k,j}z+p_{k,j}^{(1)}$, the first order term can be estimated as follows:
	\begin{equation}\label{K3-first-1}
	\int_{B(p_{k,j}^{(1)},r)}f_ke^{\phi_j}\langle \nabla(\log h_j+\phi_{k,j})(p_{k,j}^{(1)}),x-p_{k,j}^{(1)}\rangle{\rm d}x=O(e^{-(\frac{1}{2(1+\alpha_j)}+\frac{1}{1+\alpha_1})\lambda_{k,j}^{(1)}}).
	\end{equation}
	For the second order term that contains $D^2(\log\tilde{h}_j+\phi_{k,j})(p_{k,j}^{(1)})$, we have
	\begin{align*}
	&\int_{B(p_{k,j}^{(1)},r)}f_ke^{\phi_j} \langle D^2(\log \tilde{h}_j+\phi_{k,j})(p_{k,j}^{(1)})(x-p_{k,j}^{(1)}),x-p_{k,j}^{(1)}\rangle{\rm d}x  \\
	=&\int_{B(p_{k,j}^{(1)},r)}\frac{\rho_kh_j(p_{k,j}^{(1)})e^{\lambda_{k,j}^{(1)}}|x-p_{k,j}^{(1)}|^{2\alpha_j}}{\Big(1+\frac{\rho_kh_j(p_{k,j}^{(1)})}{8(1+\alpha_j)^2}e^{\lambda_{k,j}^{(1)}}|x-p_{k,j}^{(1)}|^{2(1+\alpha_j)}\Big)^2}\big(1+O(|x-p_{k,j}^{(1)}|+\epsilon_{k,j}+\epsilon_{k,1}^2)\big)  \\
	&\times\varsigma_k(x)(1+o(1))\langle D^2(\log \tilde{h}_j+\phi_{k,j})(p_{k,j}^{(1)})(x-p_{k,j}^{(1)}),x-p_{k,j}^{(1)}\rangle{\rm d}x   \\
	=&\epsilon_{k,j}^2\int_{|z|<\frac{r}{\epsilon_{k,j}}}\frac{8(1+\alpha_j)^2|z|^{2\alpha_j}}{(1+|z|^{2(1+\alpha_j)})^2}\langle D^2(\log \tilde{h}_j+\phi_{k,j})(p_{k,j}^{(1)})z,z\rangle\varsigma_{k,j}(z){\rm d}z+o(\epsilon_{k,j}^2).
	\end{align*}
	The expression above can be greatly simplified by this beautiful identity:
	\begin{equation*}
	\int_{0}^{\infty}\frac{8(1+\alpha_j)^2s^{2\alpha_j+3}}{(1+s^{2(1+\alpha_j)})^2}\frac{1-s^{2(1+\alpha_j)}}{1+s^{2(1+\alpha_j)}}{\rm d}s=-\frac{4\pi}{(1+\alpha_j)\sin\frac{\pi}{1+\alpha_j}}.
	\end{equation*}
	Consequently, Lemma \ref{lem-limit-1} and the two identities above lead to
	\begin{equation}\label{K3-second-1}
	\begin{split}
	&\int_{B(p_{k,j}^{(1)},r)}f_ke^{\phi_j} \langle D^2(\log \tilde{h}_j+\phi_{k,j})(p_{k,j}^{(1)})(x-p_{k,j}^{(1)}),x-p_{k,j}^{(1)}\rangle{\rm d}x  \\
	=&\epsilon_{k,j}^2\pi\Delta(\log \tilde{h}_j+\phi_{k,j})(p_{k,j}^{(1)})b_0\int_{0}^{\infty}\frac{8(1+\alpha_j)^2s^{2\alpha_j+3}}{(1+s^{2(1+\alpha_j)})^2}\frac{1-s^{2(1+\alpha_j)}}{1+s^{2(1+\alpha_j)}}{\rm d}s+o(\epsilon_{k,j}^2)  \\
	=&-\frac{4\pi^2\big[\Delta \log h(p_j)+\rho_*-N^*-2K(p_j)\big]}{\big(\rho_kh_j(p_{k,j}^{(1)})\big)^{\frac{1}{1+\alpha_1}}(1+\alpha_1)\sin \frac{\pi}{1+\alpha_1}}b_0e^{-\frac{\lambda_{k,j}^{(1)}}{1+\alpha_1}}+o(e^{-\frac{\lambda_{k,j}^{(1)}}{1+\alpha_1}}).
	\end{split}
	\end{equation}
	Also elementary estimate gives
	\begin{equation}\label{K3-third-1}
	\begin{split}
	&\int_{B(p_{k,j}^{(1)},r)}f_ke^{\phi_j}\langle O(|x-p_{k,j}^{(1)}|^2),x-p_{k,j}^{(1)}\rangle{\rm d}x  \\
	=&O(1)e^{-\frac{3}{2(1+\alpha_j)}\lambda_{k,j}^{(1)}}\int_{|z|<\frac{r}{\epsilon_{k,j}}}\frac{|z|^{2\alpha_j+3}}{(1+|z|^{2(1+\alpha_j)})^2}{\rm d}z \\
	=& O(1)\big(e^{-\frac{3}{2(1+\alpha_j)}\lambda_{k,j}^{(1)}}+e^{-\lambda_{k,j}^{(1)}}\big) =	o(e^{-\frac{\lambda_{k,j}^{(1)}}{1+\alpha_j}}).	
	\end{split}
	\end{equation}
	Therefore, we complete the proof of (\ref{PI-1-r-1}) by using (\ref{K1}) (\ref{K2}) and (\ref{K3-first-1})$\sim$(\ref{K3-third-1}). Note that the leading term in the second order term is ignored at this stage, since the requirement of error in the current step is still crude.
	
\smallskip
	$\mathbf{Case\ 2:}$ $t+1\leq j\leq \tau$ $(0<\alpha_j<\alpha_1)$. For the first term it is easy to see that
    \begin{equation}\label{K3-first-2}
    \int_{B(p_{k,j}^{(1)},r)}f_ke^{\phi_j} \langle \nabla(\log h_j+\phi_{k,j})(p_{k,j}^{(1)}),x-p_{k,j}^{(1)}\rangle{\rm d}x=o(e^{-\frac{\lambda_{k,j}^{(1)}}{2(1+\alpha_j)}}).
    \end{equation}
    For the second order term we have
    \begin{equation}\label{K3-second-2}
    \begin{split}
    &\int_{B(p_{k,j}^{(1)},r)}f_ke^{\phi_j} \langle D^2(\log \tilde{h}_j+\phi_{k,j})(p_{k,j}^{(1)})(x-p_{k,j}^{(1)}),x-p_{k,j}^{(1)}\rangle{\rm d}x\\
    =&O(e^{-\frac{\lambda_{k,j}^{(1)}}{1+\alpha_j}})=o(e^{-\frac{\lambda_{k,j}^{(1)}}{1+\alpha_1}}),
    \end{split}
    \end{equation}
    where we used the scaling $x=\epsilon_{k,j}z+p_{k,j}^{(1)}$ and $\alpha_j<\alpha_1$. Similar to (\ref{K3-third-1}), we know
    \begin{equation}\label{K3-third-2}
    \begin{split}
    \int_{B(p_{k,j}^{(1)},r)}f_ke^{\phi_j}\langle O(|x-p_{k,j}^{(1)}|^2),x-p_{k,j}^{(1)}\rangle{\rm d}x=o(e^{-\frac{\lambda_{k,j}^{(1)}}{1+\alpha_j}}).	
    \end{split}
    \end{equation}
    Therefore (\ref{PI-1-r-2}) follows from (\ref{K1}), (\ref{K2}) and (\ref{K3-first-2})$\sim$(\ref{K3-third-2}).

\smallskip
	$\mathbf{Case\ 3:}$ $\tau+1\leq j\leq m$ $(\alpha_j=0)$. In view of (\ref{first-deriv-est}), we get
	\begin{equation*}
		\nabla(\log h_j+\phi_{k,j})(p_{k,j}^{(1)})=\nabla(\log h+G_j^*)(p_j)+O(e^{-\frac{\lambda_{k,j}^{(1)}}{1+\alpha_1}})=O(e^{-\frac{\lambda_{k,j}^{(1)}}{1+\alpha_1}}).
	\end{equation*}
	The first order term is rather small:
	\begin{equation}\label{K3-first-3}
	\int_{B(p_{k,j}^{(1)},r)}f_ke^{\phi_j} \langle D(\log h_j+\phi_{k,j})(p_{k,j}^{(1)}),x-p_{k,j}^{(1)}\rangle{\rm d}x=O(e^{-(\frac{1}{2}+\frac{1}{1+\alpha_1})\lambda_{k,j}^{(1)}}).
	\end{equation}
	For the second order term we have
	\begin{align*}
	&\int_{B(p_{k,j}^{(1)},r)}f_ke^{\phi_j} \langle D^2(\log \tilde{h}_j+\phi_{k,j})(p_{k,j}^{(1)})(x-p_{k,j}^{(1)}),x-p_{k,j}^{(1)}\rangle{\rm d}x  \\
	=&\epsilon_{k,j}^2\int_{|z|<\frac{r}{\epsilon_{k,j}}}\frac{8}{(1+|z|^{2})^2}\langle D^2(\log \tilde{h}_j+\phi_{k,j})(p_{k,j}^{(1)})z,z\rangle \varsigma_{k,j}(z){\rm d}z+O(e^{-\lambda_{k,j}^{(1)}})  \\
	=&\epsilon_{k,j}^2\pi\Delta(\log \tilde{h}_j+\phi_{k,j})(p_{k,j}^{(1)})b_0\int_{0}^{\frac{r}{\epsilon_{k,j}}}\frac{8s^3}{(1+s^{2})^2}\frac{1-s^{2}}{1+s^{2}}{\rm d}s+O(e^{-\lambda_{k,j}^{(1)}}),
	\end{align*}
	where we have used Lemma \ref{lem-limit-1} and symmetry. It is easy to see
	\begin{equation*}
		\int_{0}^R\frac{8s^3}{(1+s^{2})^2}\frac{1-s^{2}}{1+s^{2}}{\rm d}s=-4\Big(\log (1+R^2)-\frac{1}{1+R^2}+\frac{1}{(1+R^2)^2}\Big),
	\end{equation*}
	and
	\begin{equation}\label{K3-second-3}
	\begin{split}
	\int_{B(p_{k,j}^{(1)},r)}f_ke^{\phi_j}\langle D^2(\log \tilde{h}_j+\phi_{k,j})(p_{k,j}^{(1)})(x-p_{k,j}^{(1)}),x-p_{k,j}^{(1)}\rangle{\rm d}x=O(\lambda_{k,j}^{(1)}e^{-\lambda_{k,j}^{(1)}}).
	\end{split}
	\end{equation}
	
	Finally, by scaling we immediately observe that
	\begin{equation}\label{K3-third-3}
	\begin{split}
	&\int_{B(p_{k,j}^{(1)},r)}f_ke^{\phi_j}\langle O(|x-p_{k,j}^{(1)}|^2),x-p_{k,j}^{(1)}\rangle{\rm d}x   \\
	=&O(e^{-\frac{3}{2}\lambda_{k,j}^{(1)}})\int_{|z|<\frac{r}{\epsilon_{k,j}}}\frac{|z|^{3}}{(1+|z|^{2})^2}{\rm d}z=O(e^{-\lambda_{k,j}^{(1)}}).
	\end{split}
	\end{equation}
	Therefore, $K_3$ is small in this case as well.
	\begin{equation}\label{K3-est-2}
	K_3=o(e^{-\frac{\lambda_{k,j}^{(1)}}{1+\alpha_1}}),\quad \tau+1\leq j\leq m,
	\end{equation}
	where $\alpha_1>0$ is used. Lemma \ref{lem-PI1-right} is established.

\end{proof}

Since $|A_{k,j}|=O(1)$, (\ref{PI-1}) along with Lemma \ref{lem-PI1-left} and Lemma \ref{lem-PI1-right} implies the initial estimate for $A_{k,j}$.

\begin{cor}\label{cor-A-kj}
	\begin{equation}\label{A-kj-est}
	|A_{k,j}|=o(e^{-\frac{\lambda_{k,j}^{(1)}}{2(1+\alpha_1)}}),\quad 1 \leq j \leq m.
	\end{equation}
	\begin{flushright}
		\qed
	\end{flushright}
\end{cor}

Based on (\ref{A-kj-est}), we can improve the estimates in (\ref{PI-1-l}) and (\ref{GRF-est-1}):
\begin{equation}\label{PI-1-l-re}
	{\rm (LHS)}\ {\rm of}\  (\ref{PI-1})=-4(1+\alpha_j)A_{k,j}+o(e^{-\frac{\lambda_{k,j}^{(1)}}{1+\alpha_1}})\quad 1\leq j\leq m.	
\end{equation}
\begin{equation}\label{GRF-est-2}
	\varsigma_k(x)-\bar{\varsigma}_k=o(e^{-\frac{\lambda_{k,1}^{(1)}}{2(1+\alpha_1)}}) \quad {\rm in}\ \; C^1\Big(M\setminus\bigcup_{j=1}^m B(p_{k,j}^{(1)},\theta)\Big).
\end{equation}

The identity (\ref{GRF-est-2}), which is the refined $C^1$-estimate of $\varsigma_k$ away from the blowup points, will help to improve the estimate of RHS of Pohozaev-type identity (\ref{PI-1}) and the estimate of $D\varsigma_k$. The later one will play a part in section \ref{pf-uni-2}. In order to achieve this goal, we analyse the projections of $\varsigma_{k,j}$ in more detail.

For $1\leq j\leq \tau$, we recall the equation of $\varsigma_k$ in $B(p_{k,j}^{(1)},r_0)$:
\begin{equation*}
\begin{split}
\left\{
\begin{array}{lcl}
\Delta \varsigma_k+\rho_k\tilde{h}_j|x-p_{k,j}^{(1)}|^{2\alpha_j}e^{U_{k,j}^{(1)}+G_{k,j}^{(1)}-G_{k,j}^{(1)}(p_{k,j}^{(1)})+\eta_{k,j}^{(1)}}\varsigma_k\frac{1-e^{u_k^{(2)}-u_k^{(1)}}}{u_k^{(1)}-u_k^{(2)}}=0, \\
|\varsigma_k|\leq 1,
\end{array}
\right.
\end{split}
\end{equation*}
and set the following quantities for convenience:
\begin{align*}
    &a_{k,j}=\nabla(\log h_j+G_{k,j}^{(1)})(p_{k,j}^{(1)}), \quad d_k=\parallel u_k^{(1)}-u_k^{(2)}\parallel_{L^{\infty}(M)},  \\
	&n_0=\max\big\{n\in\mathbb{N}:n\leq\frac{1}{2\gamma} \big\},\quad U_{j}(r)=\log\frac{8(1+\alpha_j)^2}{(1+r^{2(1+\alpha_j)})^2}.
\end{align*}
Then the equation for $\varsigma_k$ becomes
\begin{equation*}
\Delta \varsigma_k+\frac{\rho_kh_j(p_{k,j}^{(1)})e^{\lambda_{k,j}^{(1)}}|x-p_{k,j}^{(1)}|^{2\alpha_j}e^{g_k(x)}}{\big(1+\frac{\rho_kh_j(p_{k,j}^{(1)})}{8(1+\alpha_j)^2}e^{\lambda_{k,j}^{(1)}}|x-p_{k,j}^{(1)}|^{2(1+\alpha_j)}\big)^2}\Big\{\sum_{n=0}^{n_0}\frac{(-d_k)^{n}}{(n+1)!}\varsigma_k^{n+1}+O(d_k^{n_0+1})\Big\}=0,
\end{equation*}
where
\begin{align*}
	g_k(x)=&\langle a_{k,j},x-p_{k,j}^{(1)}\rangle\Big\{1-\frac{2(1+\alpha_j)}{\alpha_j}\Big(1+\frac{\rho_kh_j(p_{k,j}^{(1)})}{8(1+\alpha_j)^2}e^{\lambda_{k,j}^{(1)}}|x-p_{k,j}^{(1)}|^{2(1+\alpha_j)}\Big)^{-1}\Big\}  \\
	&+d_j\log\Big(2+e^{\frac{\lambda_{k,j}^{(1)}}{2(1+\alpha_j)}}|x-p_{k,j}^{(1)}|\Big)e^{-\frac{\lambda_{k,j}^{(1)}}{1+\alpha_j}}+O(|x-p_{k,j}^{(1)}|^2)+O(\epsilon_{k,1}^2)
\end{align*}
After scaling $x=\epsilon_{k,j}z+p_{k,j}^{(1)}$, we have
\begin{align}\label{equ-varsigma_kj}
\Delta \varsigma_{k,j}(z)+\frac{8(1+\alpha_j)^2|z|^{2\alpha_j}}{(1+|z|^{2(1+\alpha_j)})^2}\varsigma_{k,j}(z)=E_{k,j}(z),
\end{align}
where
\begin{align*}
	&E_{k,j}(z)=\frac{8(1+\alpha_j)^2|z|^{2\alpha_j}}{(1+|z|^{2(1+\alpha_j)})^2}
	\bigg\{-\epsilon_{k,j}\langle a_{k,j},z\rangle\big(1-\frac{2(1+\alpha_j)}{\alpha_j}\frac{1}{1+|z|^{2(1+\alpha_j)}}\big)\varsigma_{k,j}(z)  \\
	&+
	\Big[1+\epsilon_{k,j}\langle a_{k,j},z\rangle\big(1-\frac{2(1+\alpha_j)}{\alpha_j}\frac{1}{1+|z|^{2(1+\alpha_j)}}\big)\Big]\sum_{n=1}^{n_0}\frac{(-1)^{n+1}d_k^{n}}{(n+1)!}\varsigma_{k,j}^{n+1}+o(\epsilon_{k,1})\bigg\}.
\end{align*}

For each integer $l\ge 0$ we define the projections of frequency $l$ as
\begin{align*}
	&\xi_l(r)=\frac{1}{2\pi}\int_{0}^{2\pi}\varsigma_{k,j}(r\cos\theta,r\sin\theta)\cos(l\theta){\rm d}\theta,   \\
	&\tilde{\xi}_l(r)=\frac{1}{2\pi}\int_{0}^{2\pi}\varsigma_{k,j}(r\cos\theta,r\sin\theta)\sin(l\theta){\rm d}\theta.
\end{align*}
Obviously the study of $\xi_l$ is representative enough. (\ref{equ-varsigma_kj}) shows that $\xi_l$ satisfies
\begin{align*}
	\xi_l^{''}+\frac{1}{r}\xi_l^{'}+\Big(r^{2\alpha_j}e^{U_j}-\frac{l^2}{r^2}\Big)\xi_l=\tilde{E}_{l}(r),\quad l\ge 1,
\end{align*}
where
\begin{align*}
	&\tilde{E}_{1}(r)=r^{2\alpha_j}e^{U_j}\Big\{-\frac{a_{k,j}^1}{4}\epsilon_{k,j}r\big(1-\frac{2(1+\alpha_j)}{\alpha_j}\frac{1}{1+r^{2(1+\alpha_j)}}\big)\xi_0+O(d_k\xi_1)+o(\epsilon_{k,1})\Big\},  \\
	&\tilde{E}_{2}(r)=r^{2\alpha_j}e^{U_j}\Big\{-\frac{a_{k,j}^1}{4}\epsilon_{k,j}r\big(1-\frac{2(1+\alpha_j)}{\alpha_j}\frac{1}{1+r^{2(1+\alpha_j)}}\big)\xi_1+O(d_k\xi_2)+o(\epsilon_{k,1})\Big\},  \\
	&\tilde{E}_{l}(r)=r^{2\alpha_j}e^{U_j}\Big\{O(d_k\xi_l)+o(\epsilon_{k,1})\Big\},\qquad l\ge 3,
\end{align*}
and $a_{k,j}^1$ is the first component of $a_{k,j}$. Moreover, from (\ref{GRF-est-2}) we obtain that $\xi_l(0)=o(1)$ for all $l\ge 1$ and
\begin{equation}\label{ode-boundary}
\begin{split}
&|\xi_l(r)|\leq 1,\  \ \quad r\in(0,\frac{r_0}{\epsilon_{k,j}}),\ \ \quad l\ge 0,  \\
&\xi_l(r)=o(\epsilon_{k,1}),\quad r\sim e^{\frac{\lambda_{k,j}^{(1)}}{2(1+\alpha_j)}},\quad l\ge 1.
\end{split}
\end{equation}
From the equation of $\xi_l$ and the maximum principle, we only need to consider the finite $l$. Without loss of generality, we consider $1\leq l\leq l_0$ in the following analysis. Let us denote $\delta_{l,j}=\frac{l}{1+\alpha_j}$ and consider the homogeneous ordinary differential equation
\begin{equation}\label{ode}
	\xi_l^{''}+\frac{1}{r}\xi_l^{'}+\Big(r^{2\alpha_j}e^{U_j}-\frac{l^2}{r^2}\Big)\xi_l=0.
\end{equation}
By direct computation, we can verify that the following two functions are two fundamental solutions of (\ref{ode})
\begin{align*}
&\xi_{l,1}(r)=\frac{(\delta_{l,j}+1)r^l+(\delta_{l,j}-1)r^{2(1+\alpha_j)+l}}{1+r^{2(1+\alpha_j)}},\\
&\xi_{l,2}(r)=\frac{(\delta_{l,j}+1)r^{2(1+\alpha_j)-l}+(\delta_{l,j}-1)r^{-l}}{1+r^{2(1+\alpha_j)}}.
\end{align*}
Using $|\xi_l|\leq 1$ we have $C_{l,2}=0$, that is
\begin{equation*}
	\xi_l(r)=C_{l,1}\xi_{l,1}(r)+\xi_{l,p}(r)
\end{equation*}
where $C_{l,1}$ is a constant, and
\begin{equation}\label{solution-p}
\xi_{l,p}(r)=\Big(\int\frac{w_1}{w}{\rm d}r\Big)\xi_{l,1}(r)+\Big(\int\frac{w_2}{w}{\rm d}r\Big)\xi_{l,2}(r)
\end{equation}
for
\begin{equation*}
w=
\begin{vmatrix}
\xi_{l,1} & \xi_{l,2} \\
\xi_{l,1}^{'} & \xi_{l,2}^{'}
\end{vmatrix},\quad
w_1=\begin{vmatrix}
0 & \xi_{l,2} \\
\tilde{E}_l & \xi_{l,2}^{'}
\end{vmatrix},\quad
w_2=\begin{vmatrix}
\xi_{l,1} & 0 \\
\xi_{l,1}^{'} & \tilde{E}_l
\end{vmatrix}.\quad
\end{equation*}

It is easy to see that $w^{'}=(\xi_{l,1}\xi_{l,2}^{'}-\xi_{l,1}^{'}\xi_{l,2})^{'}=-\frac{1}{r}w$, which means $w(r)\sim \frac{1}{r}$. Next, let us estimate $\xi_{l}$ in $(0,\frac{r_0}{\epsilon_{k,j}})$ for $l\ge 1$.

\smallskip
For $1\leq j\leq t$, the assumption $D(\mathbf{p})=0$ implies $a_{k,j}=O(\epsilon_{k,1}^2)$. Furthermore, for $t+1\leq j\leq \tau$, it is easy to see that $\epsilon_{k,j}=o(\epsilon_{k,1})$. Therefore, for all $1\leq j\leq \tau$, we estimate $\tilde{E}_l$ as follows
\begin{equation}\label{E-1}
\begin{split}
&\tilde{E}_{l}(r)=r^{2\alpha_j}e^{U_j}\big\{o(\epsilon_{k,1})r+O(d_k\xi_l)+o(\epsilon_{k,1})\big\},\quad l=1,2;  \\
&\tilde{E}_{l}(r)=r^{2\alpha_j}e^{U_j}\big\{O(d_k\xi_l)+o(\epsilon_{k,1})\big\},\, \qquad\quad\qquad l\ge 3.
\end{split}
\end{equation}
Roughly,
\begin{equation}\label{E-2}
\begin{split}
&\tilde{E}_{l}(r)=r^{2\alpha_j}e^{U_j}\big\{o(\epsilon_{k,1})r+O(d_k)+o(\epsilon_{k,1})\big\},\quad l=1,2;  \\
&\tilde{E}_{l}(r)=r^{2\alpha_j}e^{U_j}\big\{O(d_k)+o(\epsilon_{k,1})\big\},\, \qquad\quad\qquad l\ge 3.
\end{split}
\end{equation}
By using the above estimates (\ref{E-2}) for $\tilde{E}_l$ and (\ref{solution-p}), we have
\begin{align*}
\xi_{l,p}(r)=&\big(O(d_k)+o(\epsilon_{k,1})\big)\bigg\{\Big(\int_{r}^{\infty}s^{2\alpha_j+1}e^{U_j(s)}(s+1)\xi_{l,2}(s){\rm d}s\Big)\xi_{l,1}(r)\\
&+\Big(\int_{r}^{\infty}s^{2\alpha_j+1}e^{U_j(s)}(s+1)\xi_{l,1}(s){\rm d}s\Big)\xi_{l,2}(r)\bigg\},\quad l=1,2.\\
\xi_{l,p}(r)=&\big(O(d_k)+o(\epsilon_{k,1})\big)\bigg\{\Big(\int_{r}^{\infty}s^{2\alpha_j+1}e^{U_j(s)}\xi_{l,2}(s){\rm d}s\Big)\xi_{l,1}(r)\\
&+\Big(\int_{r}^{\infty}s^{2\alpha_j+1}e^{U_j(s)}\xi_{l,1}(s){\rm d}s\Big)\xi_{l,2}(r)\bigg\},\quad l\ge 3.
\end{align*}
Direct computation shows, for $0<r<1$
\begin{align*}
	\xi_{l,p}(r)=\big(O(d_k)+o(\epsilon_{k,1})\big)\big(r^l+r^{2\alpha_j+2}\big),\quad l\ge 1;
\end{align*}
and for $1<r<\frac{r_0}{\epsilon_{k,j}}$
\begin{align*}
\xi_{l,p}(r)=\left\{
\begin{array}{lcl}
\big(O(d_k)+o(\epsilon_{k,1})\big)\big(r^{-l}+r^{-(2\alpha_j+1)}\big),\quad l=1,2, \\
\big(O(d_k)+o(\epsilon_{k,1})\big)\big(r^{-l}+r^{-(2\alpha_j+2)}\big),\quad l\ge 3.
\end{array}
\right.
\end{align*}

Consequently
\begin{equation}\label{xi-est-1}
	\parallel \xi_{l,p}\parallel_{L^{\infty}((0,\frac{r_0}{\epsilon_{k,j}}))}=O(d_k)+o(\epsilon_{k,1}),\quad 1\leq l\leq l_0.
\end{equation}
Combining (\ref{E-1}) and (\ref{xi-est-1}), we rewrite $\tilde{E}_l$ as
\begin{equation}\label{E-3}
\begin{split}
&\tilde{E}_{l}(r)=r^{2\alpha_j}e^{U_j}\big\{o(\epsilon_{k,1})r+O(d_k^2)+o(\epsilon_{k,1})\big\},\quad l=1,2;  \\
&\tilde{E}_{l}(r)=r^{2\alpha_j}e^{U_j}\big\{O(d_k^2)+o(\epsilon_{k,1})\big\},\, \qquad\quad\qquad l\ge 3.
\end{split}
\end{equation}
Then repeating the above argument $n_0$ times, we obtain
\begin{equation}\label{xi-est-2}
\parallel \xi_{l,p}\parallel_{L^{\infty}((0,\frac{r_0}{\epsilon_{k,j}}))}=o(\epsilon_{k,1}),\quad 1\leq l\leq l_0.
\end{equation}
Then, from (\ref{ode-boundary}), we have
\begin{equation*}
   C_{l,1}=o(\epsilon_{k,1})O(\epsilon_{k,j}^l).
\end{equation*}
and
\begin{equation}\label{xi-cos}
	\parallel \xi_l\parallel_{L^{\infty}((0,\frac{r_0}{\epsilon_{k,j}}))}=o(\epsilon_{k,1}),\quad  l\ge 1.
\end{equation}
Similarly,
\begin{equation}\label{xi-sin}
\parallel \tilde{\xi}_l\parallel_{L^{\infty}((0,\frac{r_0}{\epsilon_{k,j}}))}=o(\epsilon_{k,1}),\quad  l\ge 1.
\end{equation}

In other words, all projections of high frequency of $\varsigma_{k,j}$ $(1\leq j\leq \tau)$ are $o(\epsilon_{k,1})$.

\medskip
Using (\ref{xi-cos}) and (\ref{xi-sin}), we now obtain an important sharper estimate of the right hand side of (\ref{PI-1}).

\begin{lem}\label{lem-PI1-right-re}
	\begin{equation}\label{PI-1-r-4}
	{\rm (RHS)}\ {\rm of} (\ref{PI-1})=-2(1+\alpha_j)A_{k,j}+o(e^{-\frac{\lambda_{k,j}^{(1)}}{1+\alpha_1}}),\quad t+1\leq j\leq \tau.
	\end{equation}
\end{lem}

\begin{proof}[\textbf{Proof}]
	
	In view of (\ref{K1})$\sim$(\ref{expansion}) and (\ref{K3-first-2})$\sim$(\ref{K3-third-2}), we only need to improve the estimate in (\ref{K3-first-2}). In other words, it is enough to prove the following estimate
	\begin{equation}\label{K3-first-2-re}
	\int_{B(p_{k,j}^{(1)},r)}f_ke^{\phi_j}\langle \nabla(\log h_j+\phi_{k,j})(p_{k,j}^{(1)}),x-p_{k,j}^{(1)}\rangle{\rm d}x =o(e^{-\frac{\lambda_{k,j}^{(1)}}{1+\alpha_1}}).
	\end{equation}
	In fact, by the change of variable $x=\epsilon_{k,j}z+p_{k,j}^{(1)}$, we have
    \begin{align*}
    &\int_{B(p_{k,j}^{(1)},r)}f_ke^{\phi_j} \langle \nabla(\log h_j+\phi_{k,j})(p_{k,j}^{(1)}),x-p_{k,j}^{(1)}\rangle{\rm d}x  \\
    =&\int_{B(p_{k,j}^{(1)},r)}\rho_k\tilde{h}_j|x-p_{k,j}^{(1)}|^{2\alpha_j}e^{u_k^{(1)}}\varsigma_k\frac{1-e^{u_k^{(2)}-u_k^{(1)}}}{u_k^{(1)}-u_k^{(2)}}\langle a_{k,j},x-p_{k,j}^{(1)}\rangle{\rm d}x \\
    =&\int_{B(p_{k,j}^{(1)},r)}\frac{\rho_kh_j(p_{k,j}^{(1)})e^{\lambda_{k,j}^{(1)}}|x-p_{k,j}^{(1)}|^{2\alpha_j}e^{g_k(x)}}{\big(1+\frac{\rho_kh_j(p_{k,j}^{(1)})}{8(1+\alpha_j)^2}e^{\lambda_{k,j}^{(1)}}|x-p_{k,j}^{(1)}|^{2(1+\alpha_j)}\big)^2} \langle a_{k,j},x-p_{k,j}^{(1)}\rangle  \\
    &\times\Big\{\sum_{n=0}^{n_0}\frac{(-d_k)^{n}}{(n+1)!}\varsigma_k^{n+1}+O(d_k^{n_0+1})\Big\}{\rm d}x  \\
    =&\epsilon_{k,j}\int_{|z|<\frac{r}{\epsilon_{k,j}}}\frac{8(1+\alpha_j)^2|z|^{2\alpha_j}}{(1+|z|^{2(1+\alpha_j)})^2}\sum_{n=0}^{n_0}\frac{(-d_k)^{n}}{(n+1)!}\varsigma_{k,j}^{n+1}<a_{k,j},z>{\rm d}x+o(\epsilon_{k,1}^2).  \\
    \end{align*}
    where we used the fact $\alpha_j<\alpha_1$ and the definition of $n_0$.
	Then from symmetry and the estimates of high frequency of $\varsigma_{k,j}$, which are (\ref{xi-cos}) and (\ref{xi-sin}), we have the following estimate
	\begin{align*}
	\int_{B(p_{k,j}^{(1)},r)}f_ke^{\phi_j}\langle \nabla(\log h_j+\phi_{k,j})(p_{k,j}^{(1)}),x-p_{k,j}^{(1)}\rangle{\rm d}x=O(\epsilon_{k,j})o(\epsilon_{k,1})+o(\epsilon_{k,1}^2)
	=o(\epsilon_{k,1}^2).
	\end{align*}
Therefore, (\ref{K3-first-2-re}) holds. Finally, combining (\ref{K3-first-2-re}) with the proof of Lemma \ref{lem-PI1-right}, we obtain the esstimate (\ref{PI-1-r-4}).
	
\end{proof}

Based on the Pohozaev-type identity (\ref{PI-1}) and its refined estimates, which are (\ref{PI-1-l-re}) (\ref{PI-1-r-2}) (\ref{PI-1-r-3}) and (\ref{PI-1-r-4}), we can improve the estimate for $A_{k,j}$ and prove $b_0=0$.
\begin{cor}\label{cor-A-kj-re}
	\begin{equation}\label{A-kj-est-re}
	|A_{k,j}|=O(e^{-\frac{\lambda_{k,j}^{(1)}}{1+\alpha_1}}),\quad 1 \leq j \leq m.
	\end{equation}
	\begin{flushright}
		\qed
	\end{flushright}
\end{cor}
\begin{prop}\label{prop-b0}
	$b_0=0$. In particular, $b_{j,0}=0$, for $1\leq j\leq m$.
\end{prop}
\begin{proof}[\textbf{Proof}]
	Now the global cancellation property of $f_k$ plays a crucial role:
	\begin{equation*}
		\sum_{j=1}^{m}A_{k,j}=\int_{M}f_k{\rm d}\mu=0.
	\end{equation*}
	From (\ref{PI-1}) (\ref{PI-1-l-re}) (\ref{PI-1-r-2}) (\ref{PI-1-r-3}) and (\ref{PI-1-r-4}), we can see
	\begin{align*}
		b_0e^{-\frac{\lambda_{k,1}^{(1)}}{1+\alpha_1}}\sum_{j=1}^t\big[\Delta h(p_j)+\rho_*-N^*-2K(p_j)\big]\big(\rho_kh_j(p_{k,j}^{(1)})\big)^{\frac{1}{1+\alpha_1}}=o(e^{-\frac{\lambda_{k,1}^{(1)}}{1+\alpha_1}}).
	\end{align*}
	On the other hand, from (\ref{uk-ave-2}), it holds
	\begin{equation*}
		h_j^2(p_j)=h_1^2(p_1)e^{G_1^*(p_1)}e^{-G_j^*(p_j)}+o(1),\quad 1\leq j\leq t.
	\end{equation*}
	As a consequence, we obtain
	\begin{equation}\label{b0-est}
	\begin{split}
	e^{-\frac{G_1^*(p_1)}{1+\alpha_1}}\big(\rho_*h_1^2(p_1)\big)^{-\frac{1}{1+\alpha_1}}L(\mathbf{p})b_0e^{-\frac{\lambda_{k,1}^{(1)}}{1+\alpha_1}}=o(e^{-\frac{\lambda_{k,1}^{(1)}}{1+\alpha_1}}),
	\end{split}
	\end{equation}
	which together with the assumption $L(\mathbf{p})\neq 0$ implies $b_0=0$. In particular, $b_{j,0}=0$ for $1\leq j\leq m$.	
	
\end{proof}

\section{Proof of Theorem \ref{main-theorem}}\label{pf-uni-1}

\begin{proof}[\textbf{Proof of Theorem \ref{main-theorem} }]
Let $p_k^*$ be a maximum point of $\varsigma_k$, which says
\begin{equation}\label{max=1}
|\varsigma_k(p_k^*)|=1
\end{equation}
In view of Lemma \ref{lem-limit-2} and Proposition \ref{prop-b0}, we obtain the fact
\begin{equation*}
	\varsigma_k\to 0 \quad{\rm in}\ \; C_{loc}(M\backslash\{p_1,\cdots,p_m\}).
\end{equation*}
 Therefore,
\begin{equation}\label{pk*}
\lim\limits_{k\rightarrow\infty}p_k^*=p_j,
\end{equation}
for some $p_j\in\{p_1,\cdots,p_m\}$. Moreover, denoting $s_k=|p_k^*-p_{k,j}^{(1)}|$, by Lemma \ref{lem-limit-1} and Proposition \ref{prop-b0}, it holds
\begin{equation*}
\varsigma_{k,j}\rightarrow  0 \quad{\rm in}\ \; C_{loc}(\mathbb{R}^2).
\end{equation*}
Thus,
\begin{equation}\label{sk}
\lim\limits_{k\rightarrow\infty}\epsilon_{k,j}^{-1}s_k=+\infty
\end{equation}
Setting $\tilde{\varsigma_k}(x)=\varsigma_k(s_kx+p_{k,j}^{(1)})$, $|x|<s_k^{-1}r$, where $r>0$ small enough, then $\tilde{\varsigma_k}$ satisfies
\begin{align*}
0=&\Delta\tilde{\varsigma_k}(x)+\rho_k\tilde{h}_j(s_kx+p_{k,j}^{(1)})s_k^{2(1+\alpha_j)}|x|^{2\alpha_j}c_k(s_kx+p_{k,j}^{(1)})\tilde{\varsigma_k}(x) \\
=& \Delta\tilde{\varsigma_k}(x)+\frac{8(1+\alpha_j)^2(\epsilon_{k,j}^{-1}s_k)^{2(1+\alpha_j)}|x|^{2\alpha_j}}{\big(1+(\epsilon_{k,j}^{-1}s_k)^{2(1+\alpha_j)}|x|^{2(1+\alpha_j)}\big)^2}\big(1+O(s_k|x|)+o(1)\big).
\end{align*}

On the other hand, by (\ref{max=1}), we also have
\begin{equation}\label{scale-max=1}
\Big|\tilde{\varsigma_k}\big(\frac{p_k^*-p_{k,j}^{(1)}}{s_k}\big)\Big|=|\varsigma_k(p_k^*)|=1.
\end{equation}
In view of (\ref{sk}) and $|\tilde{\varsigma_k}|\leq 1$, we see that $\tilde{\varsigma_k}\rightarrow\tilde{\varsigma_0}$ in $C_{loc}(\mathbb{R}^2\backslash\{0\})$, where $\tilde{\varsigma_0}$ satisfies $\Delta\tilde{\varsigma_0}=0$ in $\mathbb{R}^2\backslash\{0\}$. Since $|\tilde{\varsigma_0}|\leq 1$, we have $\Delta\tilde{\varsigma_0}=0$ in $\mathbb{R}^2$. Hence $\tilde{\varsigma_0}$ is a constant.

Recalling that $\frac{|p_k^*-p_{k,j}^{(1)}|}{s_k}=1$ and (\ref{scale-max=1}), we find that $\tilde{\varsigma_0}\equiv 1$ or $\tilde{\varsigma_0}\equiv -1$. Therefore, we obtain that for $ k $ large enough
\begin{equation}\label{contra1}
|\varsigma_k(x)|\geq\frac{1}{2},\quad |x-p_{k,j}^{(1)}|\in\big(\frac{s_k}{2},2s_k\big).
\end{equation}

By using Lemma \ref{lem-limit-2}, we have
\begin{equation}\label{contra2}
\varsigma_k(x)=o(1)+o(1)\log R+O(R^{-2(1+\alpha_j)}),\quad |x-p_{k,j}^{(1)}|\in(R\epsilon_{k,j},d).
\end{equation}
for fixed $d>0$ small enough and arbitrary $R>0$ large enough.

\smallskip
However, by (\ref{sk}), $\epsilon_{k,j}\ll s_k$. Thus, $|\varsigma_k(s_k)|<\frac{1}{4}$ for $ k $ large enough, which contradicts with (\ref{contra1}). Theorem \ref{main-theorem} is established.
	
\end{proof}

\section{Proof of Theorem \ref{main-theorem-2}}\label{pf-uni-2}

In this section, we will analyse the behavior of $u_k^{(1)}$ and $u_k^{(2)}$ whose common blowup points include singular source(s) and regular point(s). So in this section $\tau<m$, $0<\alpha_j< 1$ for $1\leq j\leq \tau$ and $\alpha_j=0$ for $\tau+1\leq j\leq m$. Our argument is similar to the approach in \cite{bart-4} where all blowup points are regular points. The fact that $\varsigma_{k,j}\rightarrow 0$ in $C_{loc}(\mathbb{R}^2)$ for all $1\leq j \leq m$ plays a vital role.

\medskip
In \cite{lin-yan-uniq}, Lin-Yan obtained the following Pohozaev-type identity:
\begin{lemA}\label{Pohozave identity-2}
\cite{bart-4,lin-yan-uniq}

For $\tau+1\leq j\leq m$, it holds
\begin{align}\label{PI-2}
\begin{split}
& \int_{\partial B(p_{k,j}^{(1)},r)}\Big(<\nu,\nabla\varsigma_k>\nabla_iv_{k,j}^{(1)}+<\nu,\nabla v_{k,j}^{(2)}>\nabla_i\varsigma_k \Big) {\rm d}\sigma \\
& \quad -\frac{1}{2}\int_{\partial B(p_{k,j}^{(1)},r)}<\nabla(v_{k,j}^{(1)}+v_{k,j}^{(2)}),\nabla\varsigma_k>\frac{(x-p_{k,j}^{(1)})_i}{|x-p_{k,j}^{(1)}|} {\rm d}\sigma, \\
=\ & -\int_{\partial B(p_{k,j}^{(1)},r)}\rho_k\tilde{h}_j(x)\frac{e^{u_k^{(1)}}-e^{u_k^{(2)}}}{\parallel u_k^{(1)}-u_k^{(2)} \parallel_{L^{\infty}(M)} }\frac{(x-p_{k,j}^{(1)})_i}{|x-p_{k,j}^{(1)}|} {\rm d}\sigma \\
& \quad + \int_{B(p_{k,j}^{(1)},r)} \rho_k\tilde{h}_j(x)\frac{e^{u_k^{(1)}}-e^{u_k^{(2)}}}{\parallel u_k^{(1)}-u_k^{(2)} \parallel_{L^{\infty}(M)}} \nabla_i\big(\log \tilde{h}_j+\phi_{k,j}\big) {\rm d}x.
\end{split}
\end{align}

\end{lemA}

By Lemma 4.6 in \cite{bart-4} and Appendix D in \cite{lin-yan-uniq}, we have:
\begin{align}\label{PI-2-r}
{\rm (RHS)} \ {\rm of}\ (\ref{PI-2})=e^{-\frac{\lambda_{k,j}^{(1)}}{2}}\Big(\sum_{h=1}^2 D_{h,i}^2(\log \tilde{h}_j+\phi_{k,j} )(p_{k,j}^{(1)})b_{j,h}\Big)B_j+o(e^{-\frac{\lambda_{k,j}^{(1)}}{2}}).
\end{align}
for $i=1,2$ and $\tau+1\le j\le m$. The detail of this proof can be found in \cite{bart-4}.

The LHS of (\ref{PI-2}) boils down to sharp estimates of $\nabla v_{k,j}^{(i)}$ and $\nabla\varsigma_k$ on $\partial B(p_{k,j}^{(1)},r)$. The estimate for $\nabla v_{k,j}^{(i)}$ is established in   Lemma \ref{lem-Dv-kj}, and the following lemma provides the estimates for $\nabla\varsigma_k$ (see (\ref{Dsigma_k-1}) for comparison).

\begin{lem}\label{lem-C1-est-re}
	
	For any $\theta\in(0,r)$ small enough, it holds
	\begin{align}\label{GRF-est-re}
	\begin{split}
	\varsigma_k-\bar{\varsigma}_k=\sum_{j=\tau+1}^m e^{-\frac{\lambda_{k,j}^{(1)}}{2}}&\Big(\sum_{h=1}^2\partial_{y_h}G(y,x)\big|_{y=p_{k,j}^{(1)}}b_{j,h}\Big)B_j+o(e^{-\frac{\lambda_{k,1}^{(1)}}{2}}) \\
	 &{\rm in}\ \; C^1\Big(M\setminus\bigcup_{j=1}^m B(p_{k,j}^{(1)},\theta)\Big).
	\end{split}
	\end{align}
\end{lem}

\begin{proof}[\textbf{Proof}]
	Using the same notations in (\ref{J1+J2+J3}) and (\ref{J3}), now we only need to show
	\begin{equation*}
		J_1=o(e^{-\frac{\lambda_{k,1}^{(1)}}{2}}),\quad J_2=o(e^{-\frac{\lambda_{k,1}^{(1)}}{2}}).
	\end{equation*}
	
	Indeed, from (\ref{A-kj-est-re}) and the assumption $0< \alpha_1<1$, we have
	\begin{equation}\label{J1-re}
		J_1=\sum_{j=1}^m A_{k,j}G(p_{k,j}^{(1)},x)=O(e^{-\frac{\lambda_{k,1}^{(1)}}{1+\alpha_1}})=o(e^{-\frac{\lambda_{k,1}^{(1)}}{2}})
	\end{equation}
	
	Recall that
	\begin{align*}
		J_2=&\sum_{j=1}^\tau\int_{M_j}\big(G(y,x)-G(p_{k,j}^{(1)},x)\big)f_k(y){\rm d}\mu(y)   \\
		=&\sum_{j=1}^\tau\int_{B(p_{k,j}^{(1)},r_0)}f_k(y)e^{\phi_j(y)}\langle\partial _yG(y,x)\big|_{y=p_{k,j}^{(1)}},y-p_{k,j}^{(1)}\rangle {\rm d}y +O(e^{-\frac{\lambda_{k,1}^{(1)}}{1+\alpha_1}})
	\end{align*}
	Based on (\ref{xi-cos}) and (\ref{xi-sin}), by the method similar to the proof of (\ref{K3-first-2-re}) in Lemma \ref{lem-PI1-right-re}, we have
	\begin{equation*}
		\int_{B(p_{k,j}^{(1)},r_0)}f_k(y)e^{\phi_j(y)}\langle\partial _yG(y,x)\big|_{y=p_{k,j}^{(1)}},y-p_{k,j}^{(1)}\rangle {\rm d}y=O(\epsilon_{k,j})o(\epsilon_{k,1})+o(\epsilon_{k,1}^2)=O(\epsilon_{k,1}^2).
	\end{equation*}
	Therefore,
	\begin{equation}\label{J2-re}
		J_2=O(e^{-\frac{\lambda_{k,1}^{(1)}}{1+\alpha_1}})=o(e^{-\frac{\lambda_{k,1}^{(1)}}{2}}).
	\end{equation}
	
	Consequently (\ref{GRF-est-re}) holds in $C^1\big(M\setminus\bigcup_{j=1}^m B(p_{k,j}^{(1)},\theta)\big)$ and the gradient estimate is
\begin{equation}\label{Dsigma_k-2}
\begin{split}
\nabla\varsigma_k(x)=\sum_{j=\tau+1}^m e^{-\frac{\lambda_{k,j}^{(1)}}{2}}\nabla_x\Big(\sum_{h=1}^2\partial_{y_h}G(y,x)\big|_{y=p_{k,j}^{(1)}}b_{j,h}\Big)B_j+o(e^{-\frac{\lambda_{k,1}^{(1)}}{2}}).
\end{split}
\end{equation}
	
\end{proof}

By the improved estimates of $\nabla v_{k,j}^{(i)}$ and $\nabla\varsigma_k$ in  (\ref{Dv-kj-est}) and (\ref{Dsigma_k-2}), we can estimate the left hand of (\ref{PI-2}) just like Lemma 4.7 in \cite{bart-4} or Appendix D in \cite{lin-yan-uniq} and the result is:
\begin{equation}\label{PT-l}
\begin{split}
{\rm (LHS)} \ {\rm of}\ (\ref{PI-2})=&-8\pi\bigg\{\sum_{l\neq j}^{\tau+1,\cdots,m}e^{-\frac{\lambda_{k,l}^{(1)}}{2}}\partial_{x_i}\Big(\sum_{h=1}^2\partial_{y_h}G(y,x)\big|_{y=p_{k,l}^{(1)}}b_{l,h}\Big)B_l\\
&\ \;+e^{-\frac{\lambda_{k,j}^{(1)}}{2}}\partial_{x_i}\Big(\sum_{h=1}^2\partial_{y_h}R(y,x)\big|_{x=y=p_{k,j}^{(1)}}b_{j,h}\Big)B_j\bigg\}+o(e^{-\frac{\lambda_{k,j}^{(1)}}{2}}).
\end{split}
\end{equation}

Finally we prove $b_{j,1}=b_{j,2}=0$ for all $j$.
\begin{prop}\label{prop-b1b2}

$b_{j,1}=b_{j,2}=0$, for all $j=\tau+1,\cdots,m$. In particular,
\begin{equation*}
\varsigma_{k,j}\rightarrow 0\quad {\rm in}\ \; C_{loc}(\mathbb{R}^2),\quad {\rm for\ \, all} \ \, j=1,\cdots,m.
\end{equation*}
\end{prop}

\begin{proof}[\textbf{Proof}]
	
Obviously, (\ref{PI-2}) together with (\ref{PI-2-r}) and (\ref{PT-l}) implies, for all $i=1,2$, and $j=\tau +1,\cdots,m$,
\begin{align}\label{PI-final}
\begin{split}
&e^{-\frac{\lambda_{k,j}^{(1)}}{2}}\Big(\sum_{h=1}^2 D_{h,i}^2(\log \tilde{h}_j+\phi_{k,j})(p_{k,j}^{(1)})b_{j,h}\Big)B_j\\
=&-8\pi\sum_{l\neq j}^{\tau+1,\cdots,m}e^{-\frac{\lambda_{k,l}^{(1)}}{2}}\partial_{x_i}\Big(\sum_{h=1}^2\partial_{y_h}G(y,x)\big|_{y=p_{k,l}^{(1)}}b_{l,h}\Big)B_l\\
&-8\pi e^{-\frac{\lambda_{k,j}^{(1)}}{2}}\partial_{x_i}\Big(\sum_{h=1}^2\partial_{y_h}R(y,x)\big|_{x=y=p_{k,j}^{(1)}}b_{j,h}\Big)B_j+o(e^{-\frac{\lambda_{k,j}^{(1)}}{2}}).
\end{split}
\end{align}

Set $\vec{b}=(\tilde{b}_{\tau+1,1}B_{\tau+1},\tilde{b}_{\tau+1,2}B_{\tau+1},\cdots,\tilde{b}_{m,1}B_m,\tilde{b}_{m,2}B_m)$, where
\begin{equation*}
\tilde{b}_{l,h}=\lim\limits_{k\to +\infty}\big(e^{\frac{\lambda_{k,j}^{(1)}-\lambda_{k,1}^{(1)}}{2}}b_{l,h}\big).
\end{equation*}
Then by (\ref{p_kj-location}) and letting $k\to+\infty$, we obtain that
\begin{equation}\label{b-vector}
D^2f^*(p_{\tau+1},\cdots,p_m)\cdot \vec{b}=0
\end{equation}
By using the non-degeneracy assumption $\det \big(D^2f^*(p_{\tau+1},\cdots,p_m)\big)\neq 0$, we conclude that
\begin{equation}\label{b=0}
b_{j,1}=b_{j,2}=0,\quad j=\tau+1,\cdots,m.
\end{equation}
Proposition \ref{prop-b1b2} is established.

\end{proof}

\begin{proof}[\textbf{Proof of Theorem \ref{main-theorem-2} }]
From
Lemma \ref{lem-limit-2} and Proposition \ref{prop-b0} $\varsigma_k$ tends to $0$ in
$C_{loc}(M\backslash\{p_1,\cdots,p_m\})$.
By Lemma \ref{lem-limit-1} and Proposition \ref{prop-b1b2}, we have
\begin{equation*}
	\varsigma_{k,j}\rightarrow  0\quad  {\rm in}\ \; C_{loc}(\mathbb{R}^2),\quad 1\leq j\leq m.
\end{equation*}
Theorem \ref{main-theorem-2} follows just like the last step of the proof of  Theorem \ref{main-theorem}.

\end{proof}

Finally, we finish to prove Theorem \ref{main-theorem-3} and Theorem \ref{main-theorem-4} about Dirichlet problems.

\begin{proof}[\textbf{Proof of Theorems \ref{main-theorem-3}, \ref{main-theorem-4} }]
	For the blowup solutions to (\ref{equ-flat}), the corresponding estimates as in section \ref{preliminary} have been also obtained in \cite{chen-lin}\cite{zhang2} for $\alpha_j\in\mathbb{R}^+\setminus\mathbb{N}$ and in \cite{chen-lin-sharp}\cite{zhang1}\cite{gluck} for $\alpha_j=0$. Those preliminary estimates have almost the same form except for $\phi_j=0$ and $K\equiv 0$, where $\phi_j$ are the conformal factor at $p_j$ and $K$ is the Gaussian curvature of $M$.
	
	Then, under the assumption of regularity about $\partial\Omega$ and $q_j\in \Omega$ $(1\leq j\leq N)$, \cite{ma-wei} has showed that the blowup points of (\ref{equ-flat}) are far away from $\partial\Omega$ via the moving plane method and the Pohozaev identities. Consequently, the terms coming from the boundary of domain are included in the error term. In other words, those boundary terms do not affect our argument.
	
	On the other hand, the vital part of estimates obtained in section \ref{difference}, \ref{anal-pohozaev} and \ref{pf-uni-2} only come from local analysis, Therefore, such results still work for the Dirichlet problem (\ref{equ-flat}).
	
	Thus, Theorem \ref{main-theorem-3} and Theorem \ref{main-theorem-4} can be proved as Theorem \ref{main-theorem} and Theorem \ref{main-theorem-4}, respectively. 
	
\end{proof}


\begin{thebibliography}{99}
\bibitem{ambjorn} J. Ambjorn, P. Olesen, Anti-screening of large magnetic fields by vector bosons. \emph{Phys. Lett. B}, \textbf{214} (1988), no. 4, 565--569.
\bibitem{bart-taran-mass} D. Bartolucci, The Liouville equations with singular data: A concentration-compactness principle via alocal representation formula, \emph{J. Differential Equations} \textbf{185} (2002), 161-180.

\bibitem{bart-4} D. Bartolucci, A. Jevnikar, Y. Lee, W. Yang, Uniqueness of bubbling solutions of mean field equations. \emph{J. Math. Pures Appl.}, (9) \textbf{123} (2019), 78--126.

\bibitem{bart-4-2} D. Bartolucci, A. Jevnikar, Y. Lee, W. Yang, Local Uniqueness of Blowup solutions of mean field equations with singular data. Preprint 2019.

\bibitem{Bartolucci-Chen-Lin-T} D. Bartolucci, C. C. Chen, C. S. Lin, G. Tarantello, Profile of blow up solutions to mean field equations with singular data, \emph{Comm. in P.D.E.}, \textbf{29}(7-8) (2014), 1241-1265.

\bibitem{bart-lin} D. Bartolucci, C. S. Lin, Uniqueness Results for Mean Field Equations with Singular Data, \emph{Comm. in P. D. E.}, \textbf{34}(7) (2009), 676-702.

\bibitem{bart-taran} D. Bartolucci, G. Tarantello, Liouville type equations with singular data and their applications to periodic multivortices for the electroweak theory, \emph{Comm. Math. Phys.}, \textbf{229} (2002), 3-47.

\bibitem{bart-nonsimple}  D. Bartolucci; G. Tarantello, Asymptotic blow-up analysis for singular Liouville type equations with applications. \emph{J. Differential Equetions}, \textbf{262} (2017), 3887-3931.


\bibitem{caglioti-1} E. Caglioti, P.L. Lions, C. Marchioro , M. Pulvirenti, A special class of stationary flows for two-dimensional Euler equations: A statistical mechanics description, \emph{Comm. Math. Phys.}, \textbf{143} (1992), 501--525.

\bibitem{caglioti-2} E. Caglioti, P.L. Lions, C. Marchioro , M. Pulvirenti, A special class of stationary flows for two-dimensional Euler equations: A statistical mechanics description, part II, \emph{Comm. Math. Phys.}, \textbf{174} (1995), 229--260.



\bibitem{chai} C.C. Chai, C.S. Lin, C.L.Wang, Mean field equations, hyperelliptic curves, and modular forms: I, \emph{Camb. J. Math.}, \textbf{3}(1-2) (2015), 127-274.

\bibitem{chan-fu-lin} H. Chan, C. C. Fu, C. S. Lin, Non-topological multi-vortex solutions to the self-dual Chern-Simons-Higgs equation, \emph{Comm. Math. Phys.}, \textbf{231} (2002), no. 2, 189-221.

\bibitem{chen-lin-sharp} C. C. Chen, C. S. Lin, Sharp estimates for solutions of multi-bubbles in compact Riemann surface. \emph{Comm. Pure Appl. Math.}, \textbf{55} (2002), 728-771.

\bibitem{chen-lin-deg} C. C. Chen, C. S. Lin, Topological degree for a mean field equation on Riemann surfaces. \emph{Comm. Pure Appl. Math.}, \textbf{56} (2003), 1667-1727.

\bibitem{chen-lin-wang} C. C. Chen, C. S. Lin, G.Wang, Concentration phenomena of two-vortex solutions in a ChernSimons model. \emph{Ann. Sc. Norm. Super. Pisa Cl. Sci.}, (5) \textbf{3} (2004), 2, 367397.

\bibitem{chen-lin} C. C. Chen, C. S. Lin, Mean field equation of Liouville type with singularity data: Sharper estimates, \emph{Discrete and Continuous Dynamic Systems-A}, \textbf{28} (2010), 1237-1272

\bibitem{chen-lin-deg-2} C. C. Chen, C. S. Lin, Mean field equation of Liouville type with singular data: topological degree. \emph{Comm. Pure Appl. Math.}, \textbf{68} (2015), 6, 887-947.
\bibitem{chen-kuo-lin} Z. J. Chen, T.J. Kuo and C.S. Lin, Hamiltonian system for the elliptic form of Painleve VI equation, \emph{J. Math. Pure App.}, \textbf{106}(3) (2016), 546-581.
\bibitem{coddington} E. A. Coddington, N. Levinson, Theory of ordinary differential equations. McGraw-Hill Book Company, Inc., New York-Toronto-London, 1995.
\bibitem{polia-taran} A. Poliakovsky, G. Tarantello, On a planar Liouville-type problem in the study of selfgravitating strings, \emph{J. Differential Equations}, \textbf{252} (2012), 3668-3693.

\bibitem{fang-lai} H. Fang, M. Lai, On curvature pinching of conic 2-spheres, \emph{Calc. Var.  P.D.E.}, \textbf{55} (2016), 118.
\bibitem{gluck} M. Gluck, Asymptotic behavior of blow up solutions to a class of prescribing Gauss curvature equations. \emph{Nonlinear Anal.}, \textbf{75} (2012), 5787-5796.

\bibitem{del-pino-1} M. Kowalczyk, M. Musso , M. del Pino, Singular limits in Liouville-type equations, \emph{Calc. Var.  P.D.E.}, \textbf{24}(1) (2005), 47-81.
\bibitem{kuo-lin} T. J. Kuo, C.S. Lin, Estimates of the mean field equations with integer singular sources: non-simple blow up, \emph{Jour. Diff. Geom.}, \textbf{103} (2016), 377-424.

\bibitem{li-cmp} Y.Y. Li, Harnack type inequality: the method of moving planes, \emph{Comm. Math. Phys.}, \textbf{200} (1999), 421--444.

\bibitem{lin-yan-uniq}  C. S. Lin, S. S. Yan, On the mean field type bubbling solutions for Chern-Simons-Higgs equation. \emph{Adv. Math.}, \textbf{338} (2018), 1141--1188.
\bibitem{ma-wei} Li Ma, Juncheng Wei, Convergence for a Liouville equation. Comment. Math. Helv. 76 (2001), no. 3, 506–514.
\bibitem{machiodi-1}A. Malchiodi, D. Ruiz,
New improved Moser-Trudinger inequalities and singular Liouville equations on compact surfaces.
Geom. Funct. Anal. 21 (2011), no. 5, 1196–1217.
\bibitem{nolasco-taran} M. Nolasco, G. Tarantello, On a sharp Sobolev-type Inequality on two-dimensional compact manifold, \emph{Arch. Rat. Mech. An.}, \textbf{145} (1998), 161-195.
\bibitem{Parjapat-Tarantello} J. Prajapat, G. Tarantello, On a class of elliptic problems in $\mathbb{R}^2$: Symmetry and uniqueness results, \emph{Proc. Roy. Soc. Edinburgh Sect. A}, \textbf{131} (2001), 967-985.

\bibitem{spruck-yang} J. Spruck, Y. Yang, J. Spruck, Y. Yang, On Multivortices in the Electroweak Theory I:Existence of Periodic Solutions, \emph{Comm. Math. Phys.}, \textbf{144} (1992), 1-16.
\bibitem{suzuki} T. Suzuki, Global analysis for a two-dimensional elliptic eiqenvalue problem with the exponential nonlinearly, \emph{Ann. Inst. H. Poincare Anal. Nonlinear }, \textbf{9}(4) (1992), 367-398.
\bibitem{taran-1} G. Tarantello, Multiple condensate solutions for the Chern-Simons-Higgs theory, \emph{J. Math. Phys.}, \textbf{37} (1996), 3769-3796.

\bibitem{taran-2} G. Tarantello, Self-Dual Gauge Field Vortices: An Analytical Approach, PNLDE 72, Birkhauser Boston, Inc., Boston, MA, 2007.

\bibitem{taran-3} G. Tarantello, Blow-up analysis for a cosmic strings equation, \emph{Jour. Funct. An.}, \textbf{272} (1) (2017) 255-338.

\bibitem{troy} M. Troyanov, Prescribing curvature on compact surfaces with conical singularities, \emph{Trans. Amer. Math. Soc.}, \textbf{324} (1991), 793-821.

\bibitem{wei-zhang-pacific} Juncheng Wei, Lei Zhang, Nondegeneracy of the Gauss curvature equation with negative conic singularity. Pacific J. Math. 297 (2018), no. 2, 455–475.

\bibitem{wei-zhang-19} Juncheng Wei, Lei Zhang, Estimates for Liouville equation with quantized singularities, preprint, https://arxiv.org/abs/1905.04123

\bibitem{wolan} G. Wolansky, On steady distributions of self-attracting clusters under friction and fluctuations, \emph{Arch. Rational Mech. An.}, \textbf{119} (1992), 355--391.

\bibitem{y-yang} Y. Yang, Solitons in Field Theory and Nonlinear Analysis, Springer Monographs in Mathematics, Springer, New York, 2001.

\bibitem{zhang1} L. Zhang, Blowup solutions of some nonlinear elliptic equations involving exponential nonlinearities, \emph{Comm. Math. Phys}, \textbf{268} (2006), 105-133.

\bibitem{zhang2} L. Zhang, Asymptotic behavior of blowup solutions for elliptic equations with exponential nonlinearity and singular data, \emph{Commun. Contemp. Math}, \textbf{11} (2009), 395-411.

\end{thebibliography}
\end{document}